\documentclass[12pt]{amsart}
\usepackage{amssymb,amsmath}
\usepackage{geometry}    
\usepackage{mathrsfs}

\usepackage{graphicx}

\usepackage{vmargin}

\usepackage{tikz}  

\setmargrb{1in}{1in}{1in}{1in}

\newtheorem{theorem}{Theorem}[section]
\newtheorem{lemma}[theorem]{Lemma}
\newtheorem{proposition}[theorem]{Proposition}
\newtheorem{corollary}[theorem]{Corollary}
\theoremstyle{definition}
\newtheorem{definition}{Definition}
\newtheorem{remark}{Remark}
\newtheorem{example}{Example}

\newtheorem{problem}{Problem}

\usepackage{xcolor}

\newcommand{\wla}{ \widehat{\lambda}_n(\xi)  }

\newcommand{\w}{w_n(\xi) }

\newcommand{\wo}{\widehat{w}_n(\xi) }

\newcommand{\ws}{ w_{=n}^{\ast}(\xi) }

\begin{document}

\title[On Wirsing's problem in small exact degree]{On Wirsing's problem in small exact degree}

\author{Johannes Schleischitz}

\thanks{Middle East Technical University, Northern Cyprus Campus, Kalkanli, G\"uzelyurt \\
	johannes@metu.edu.tr ; jschleischitz@outlook.com}

\begin{abstract}
	We investigate a variant of Wirsing's problem on 
	approximation to a real number
	by real algebraic numbers of degree exactly $n$.
	This has been studied by Bugeaud and Teulie. We improve
	their bounds for degrees up to $n=7$. Moreover, we obtain
	results regarding small values of polynomials
	and approximation to a real number by algebraic integers 
	and units in small 
	prescribed degree. 
	The main ingredient are irreducibility criteria for
	integral linear combinations of coprime integer polynomials.
	Moreover, for cubic polynomials, these criteria improve results 
	of Gy\H{o}ry on a problem of Szegedy.
\end{abstract}

\maketitle

{\footnotesize{

{\em Keywords}: Wirsing's problem, exponents of Diophantine approximation, irreducibility of integer polynomials \\
Math Subject Classification 2010: 11J13, 11J82, 11R09}}

\vspace{1mm}

\section{Introduction and main results}

\subsection{Wirsing's problem in exact degree} \label{intro}

A classical topic that goes back to Wirsing~\cite{wirsing}
is to study approximation to a real number $\xi$ by algebraic numbers of degree at most $n$.
The classical exponent $w_{n}^{\ast}(\xi)$ introduced by Wirsing himself 
provides a measure for the approximation quality. 
It is defined as the supremum
of $w$ for which
\begin{equation}  \label{eq:alp}
0<|\xi- \alpha| \leq H(\alpha)^{-w-1}
\end{equation}
has infinitely many solutions in algebraic numbers $\alpha$ of degree
at most $n$. Here $H(\alpha)=H(P)$ is the maximal modulus
of the coefficients of the minimal polynomial $P$ of $\alpha$ over $\mathbb{Z}[T]$, with coprime coefficients.
Wirsing formulated the longstanding open question if $w_{n}^{\ast}(\xi)\geq n$
for all transcendental real numbers $\xi$. 
This is true for $n=1$ by Dirichlet's Theorem, 
and was further verified for $n=2$ by Davenport and Schmidt~\cite{davsh67}.
The best known lower bounds for $w_{n}^{\ast}(\xi)$ for larger $n$ 
are due to Tsishchanka~\cite{Tsi07}
for $n\leq 24$ and Badziahin, Schleischitz~\cite{buteu} for $n>24$.

We study approximation by algebraic numbers of exact degree $n$.
The according variant of Wirsing's problem was investigated by Bugeaud and
Teulie~\cite{buteu},
i.e. if the exponent $\ws$ defined below is at least $n$ for every 
transcendental real number $\xi$. See also~\cite[Problem~23]{bugbuch}
for the formulation of a slightly stronger claim that remains open. 

\begin{definition}
		Let $\ws$ be supremum of $w$ so that \eqref{eq:alp}
	has infinitely many solutions in algebraic numbers $\alpha$ 
	of degree precisely $n$.
\end{definition}

The following improves on~\cite{buteu} for small $n$ and is 
the main result of this section.

\begin{theorem}  \label{H}
	For $1\leq n\leq 7$ an integer and any transcendental real number $\xi$ we have 
	\begin{equation}  \label{eq:1}
	\ws \geq \frac{ n+\sqrt{n^2+16n-8}}{4}.
	\end{equation}
\end{theorem}

For the sequel we need to introduce other auxiliary classical 
exponents of approximation that are closely related to $\ws$.

\begin{definition}  \label{d2}
	Let $w_n(\xi)$ resp. $\widehat{w}_n(\xi)$ be the supremum of $w$ 
	such that the system
    \begin{equation} \label{eq:hass}
    H(P)\leq X, \qquad 0<|P(\xi)| \leq X^{-w}
    \end{equation}
	has a solution in integer polynomials of degree at most $n$
	for certain arbitrarily large $X$ and all large $X$, respectively.
	Let $\lambda_n(\xi)$ resp. $\wla$ be the supremum of $\lambda$ such that
	\[
	0<x\leq X, \qquad \max_{1\leq i\leq n} \Vert x \xi^{i}\Vert \leq X^{-\lambda} 
	\]
	has an integer solution $x$ for certain
	arbitrarily large $X$ and all large $X$, respectively, where
	$\Vert .\Vert$ denotes the distance to the nearest integer.
\end{definition}

	Variants of Dirichlet's Theorem imply for any transcendental real $\xi$
	the lower bounds
\[
w_{n}(\xi)\geq \wo \geq n, \qquad \lambda_n(\xi)\geq \wla \geq \frac{1}{n},
\]
for any $n\geq 1$. Moreover we should point out the well-known inequality
\begin{equation}  \label{eq:dau}
w_{n}(\xi)\geq w_{n}^{\ast}(\xi),\qquad\qquad n\geq 1,
\end{equation}
for any $\xi$, as follows from Proposition~\ref{pop} below. However, 
clearly these estimates do not allow for drawing 
any conclusion on Wirsing's problem and its variants.

We return to approximation in exact degree, especially our problem 
if $\ws\geq n$ holds for any transcendental real $\xi$. 
For $n=2$, it was shown in~\cite{moscj}, refining
a result of Moshchevitin~\cite{nimo} (which in
turn refined on Jarn\'ik~\cite{jarnik}) to exact degree, 
that indeed
\begin{equation}  \label{eq:jm}
w_{=2}^{\ast}(\xi) \geq \widehat{w}_{2}(\xi)\cdot (\widehat{w}_{2}(\xi)-1)\geq 2
\end{equation}
holds. The left inequality is sharp in the non-trivial case 
when $\xi$ is a so-called
extremal number~\cite{roy}. 
For any $n>2$ the problem is open (like Wirsing's original problem).
Contributions have so far been obtained by Bugeaud and Teulie~\cite{buteu}, see also~\cite{moscj, teu}.
Building up on ideas by
Davenport and Schmidt~\cite{davsh} it is implicitly shown
in~\cite{buteu} that
\begin{equation}  \label{eq:ds}
\ws \geq \frac{1}{ \wla }, \qquad \qquad n\geq 1.
\end{equation}
This motivates to study upper bounds for $\wla$.
However, this topic turned out to be 
quite challenging. 
Any irrational real $\xi$ induces 
$\widehat{\lambda}_1(\xi)=1$, see~\cite{KH}. 
While studying another variant of Wirsing's problem regarding
approximation by algebraic integers, see also \S~\ref{it2},
Davenport and Schmidt~\cite{davsh} were the first 
to systematically investigate
the exponents $\wla$ for $n\geq 2$.
For $n=2$ 
the bound 
\begin{equation} \label{eq:2la}
\widehat{\lambda}_2(\xi) \leq \frac{ \sqrt{5}-1}{2}= 0.6180\ldots
\end{equation}
from~\cite[Theorem~1a]{davsh} was verified to be sharp by Roy~\cite{roy}. 
For $n>2$, Davenport and Schmidt~\cite[Theorem~2a]{davsh} established
upper bounds for $\wla$ of order roughly $2/n$. They turned out to be no longer optimal, however only small improvements 
have been obtained so far.  
For $n=3$ see Roy's paper~\cite{roy3}. The very recent paper by
Poels and Roy~\cite{pr}, 
appearing on arXiv only after
the first version of this note, 
contains the best known bound for any $n\geq 4$,
thereby improving on intermediate work~\cite{badz, laurent, equprin, period}.
In~\cite{pr} unconditional upper bounds of order $2/n-O(n^{-2})$ were 
finally qualitatively improved for the first time, 
the new bound via \eqref{eq:ds} leads to a lower bound of the form
\begin{equation}  \label{eq:evenw}
\ws\geq \frac{n}{2}+ \frac{1-\log 2}{2}\sqrt{n}+\frac{1}{3}, \qquad n\geq 4,
\end{equation}
and stronger bounds for small $n$. These are the exact same bounds
as in~\cite{pr} for $\tau_{n+1}(\xi)-1$, 
where $\tau_{n+1}(\xi)$ is defined in~\cite{pr}.

The following table compares the bounds of Theorem~\ref{H} with those from
\eqref{eq:ds} combined with the upper bounds for $\wla$ 
from~\cite[Theorem~1.2, 1.3]{pr},~\cite{roy3}.
For sake of completeness we include
the bounds for $w_{n}^{\ast}(\xi)$ by Tsishchanka~\cite{Tsi07} 
as well, where no 
restriction to exact degree is imposed.
We cut off after $4$ decimal places. 

\begin{center}
	\begin{tabular}{ |c|c|c|c| }
		\hline
		n & Thm 1.1 & Bugeaud \& Teulie, 
		Poels \& Roy, Roy & Tsishchanka (not exact degree!) \\ \hline
		3 & 2.5 
		& 2.3557 & 2.7304\\
		4 & 3.1213 
		& 2.9667 & 3.4508 \\
		5 & 3.7122 
		& 3.5615 & 4.1389 \\
		6 & 4.2839 
		&  4.0916 & 4.7630 \\
		7 & 4.8423 
		& 4.6457 & 5.3561  \\
		\hline
	\end{tabular}
\end{center}

While~\cite{pr} shrinked the gap,
the new bounds \eqref{eq:1} in the second column 
remain reasonably stronger than those in the third column
that rely purely on $\wla$.
Thereby \eqref{eq:1} is also stronger than the best known bounds on approximation
by algebraic integers (resp. units) of degree at most $n+1$ (resp. $n+2$) from~\cite{pr},
that coincide with the third column.

The improvement in Theorem~\ref{H} relies on the following analogue
of~\cite[Theorem~2.7]{buschlei} for approximation by algebraic numbers 
of exact degree $n$.

\begin{theorem}  \label{t1}
	For an integer $1\leq n\leq 7$ and every transcendental real $\xi$ we have
	\[
	w_{=n}^{\ast}(\xi) \geq \frac{3}{2} \widehat{w}_n(\xi) - n +\frac{1}{2}.
	\]
\end{theorem}

We believe the claim remains true for all $n$.
For any $n$ where this applies, we directly infer the 
same lower bound for $w_{=n}^{\ast}(\xi)$ that had been obtained in~\cite[Theorem~1.2]{badsch} for the exponent $w_n^{\ast}(\xi)$ unconditionally,
by precisely the same line of arguments in that paper. This 
bound is of order $\ws>n/\sqrt{3}$.
From \eqref{eq:1} asymptotically 
we would only derive
a bound $\ws\geq n/2+2-o(1)$ as $n\to\infty$,
even weaker than \eqref{eq:evenw},
however for small $n$ it turns out stronger than both \eqref{eq:evenw}
and~\cite{badsch}.

Together with \eqref{eq:ds} and German's transference inequality~\cite{german}
\begin{equation}   \label{eq:w}
\wla \leq \frac{\wo - n +1}{\wo}, \qquad n\geq 1,
\end{equation}
we directly infer Theorem~\ref{H} as follows.

\begin{proof}[Deduction of Theorem~\ref{H} from Theorem~\ref{t1}]
	From Theorem~\ref{t1} and \eqref{eq:ds}, \eqref{eq:w} we get
	\[
	\ws \geq \max \left\{ \frac{1}{ \wla } \; , \;  \frac{3}{2} \widehat{w}_n(\xi) - n +\frac{1}{2} \right\}  \geq 
	\max\left\{ \frac{\wo}{\wo - n +1} , \frac{3}{2} \widehat{w}_n(\xi) - n +\frac{1}{2} \right\}.
	\]
	Since the left bound decreases whereas the right increases as
	functions of $\wo$, the equilibrium yields the smallest possible value which can be determined as given in \eqref{eq:1}.
\end{proof}

Obviously any improvement of \eqref{eq:w} directly
strengthens Theorem~\ref{H}, at least for $n\leq 7$. While inequality \eqref{eq:w} is known to be an identity
for certain vectors $\underline{\xi}=(\xi_{1},\ldots,\xi_{n})$ that are $\mathbb{Q}$-linearly independent together with $\{ 1\}$ and a very similarly defined exponent $\widehat{\lambda}(\underline{\xi})$ where $\xi^{i}$ is 
replaced by $\xi_{i}$ in Definition~\ref{d2}, see Schmidt and Summerer~\cite{ss}, it is likely that it can be sharpened in our special situation of vectors on the Veronese curve. 

\subsection{Related new results in exact degree}  \label{it2}
Our method gives rise to several other new results. 
We first define another exponent for polynomial evaluation in exact degree.

\begin{definition}  \label{d3}
Let $w_{=n}(\xi)$ be supremum of $w$ so that
\eqref{eq:hass} has a solution in irreducible integer polynomials $P(T)$ 
of degree exactly $n$, for certain arbitrarily large $X$.
\end{definition}

Similar to \eqref{eq:dau}, the standard argument Proposition~\ref{pop} 
yields for every $\xi$ the estimate
\begin{equation} \label{eq:stand}
w_{=n}(\xi)\geq \ws.
\end{equation}
As noticed in~\cite{moscj}, the exponent $w_{=n}(\xi)$ would coincide with $w_{n}(\xi)$ if we omit the irreduciblity condition on $P$ in Definition~\ref{d3}, by multiplying polynomials derived
from Dirichlet's Theorem with suitable powers of the variable $T$ if needed. 
The same holds when we do not restrict to exact degree,
see Lemma~\ref{hiw} below. However, for our exponent,
it is unclear if it is bounded below by $n$ 
for every transcendental real number $\xi$.
As a byproduct of our method, we can verify this for small $n$.

\begin{theorem}  \label{t2}
	For an integer $1\leq n\leq 7$ and every transcendental real 
	number $\xi$ we have
	\[
	w_{=n}(\xi) \geq \wo \geq n.
	\]
\end{theorem}

The bound $n$ is clearly optimal, almost all $\xi$ induce
equalities in Theorem~\ref{t2}. Examples where $w_n(\xi)>\wo$
but still the left estimate is sharp are given by
extremal numbers $\xi$ defined by Roy~\cite{roy}: they
satisfy $w_3(\xi)= w_2(\xi)= 2+\sqrt{5}$
and $w_{=3}(\xi)= \widehat{w}_{3}(\xi)=3$, see~\cite{ext}.
Similar results apply for every Sturmian continued fraction~\cite{jnt}.
Theorem~\ref{t2} was known for $n=2$ in view of \eqref{eq:jm} and \eqref{eq:stand}.
The claim for $n=3$ also occurs in~\cite{moscj}.
However, there was a mistake in the latter proof. Indeed, in the 
proof of~\cite[Theorem 3.2]{moscj}, we cannot assume that the involved 
polynomial $Q$ has degree less than $n$, as it could have degree
$n$ but be reducible. 
In this case we cannot 
deduce $w_{=3}(\xi)\geq \widehat{w}_{3}(\xi)\geq 3$.
Thus the preparatory result 
Theorem~\ref{t4} below from~\cite{moscj} in its present form
is insufficient to derive the claim.
This is a serious technical obstacle, see also
\S~\ref{intro3} below.
So we provide a new, correct proof of the case $n=3$ of Theorem~\ref{t2}
and settle the cases $n\in \{ 4,5,6,7\}$ as well in the present paper.
Our proof still uses Theorem~\ref{t4} derived in~\cite{moscj}, however
embedded in a considerably more intricate argument.

We want to point out that our proofs show that Theorems~\ref{H},~\ref{t1},~\ref{t2} hold
for all pairs $n$ and $\xi$ satisfying $\wo>2n-7$. The best known unconditional upper bound valid for any real $\xi$ is $\wo\leq 2n-2$ for $n\geq 10$, and slightly weaker bounds for smaller $n$. See~\cite{ichacta},
where also a stronger conjectural bound of order $\wo< (1+\frac{1}{\sqrt{2}})n$ was motivated.

Despite Theorem~\ref{t2}, the following refined version of
Dirichlet's Theorem already posed in~\cite[\S~6]{moscj}
remains open in exact degree $n\geq 3$.

\begin{problem}
	Given $n, \xi$, is there $c=c(n,\xi)>0$ such that $|P(\xi)|<cH(P)^{-n}$
	holds for infinitely many irreducible integer polynomials $P$ of degree exactly $n$?
\end{problem}

Again the answer is easily seen to be 
positive if we omit either irreducibility or 
the exact degree condition on $P$. 
For $n=2$ the claim follows from~\cite[Theorem~1.1]{moscj} and Proposition~\ref{pop}.
Clearly for $3\leq n\leq 7$ the problematic case is $\wo=n$, 
more precisely
the case of vectors $(\xi,\xi^2,\ldots,\xi^n)$ that are 
singular but not very singular, is open. 
Our method admits a bound $|P(\xi)|<H(P)^{-n} (\log H(P))^{-h}$ for some 
explicit not too large $h>0$, if we assume that 
the factorizations of leading and constant
coefficients of the best approximation polynomials associated to $\xi$
do not have a very biased behavior. To conclude the remarks
to Theorem~\ref{t2}, we want
to mention that the accordingly defined uniform exponent in exact degree
$\widehat{w}_{=n}(\xi)$ takes the value $0$ for $n\geq 2$ and  
$\xi$ any Liouville number~\cite[Corollary~3.10]{moscj}, showing
that there is no uniform Dirichlet Theorem in exact degree.
A similar result on the exponent $\widehat{w}_{n}^{\ast}$ is due to Bugeaud~\cite{bugdraft}.

We state some consequences of our method regarding 
approximation by algebraic {\em integers} and {\em units}
of prescribed degree,
although it is insufficient to improve the best known bounds 
originating in combination of~\cite{buteu},~\cite{pr},~\cite{roy3},~\cite{teu}.

\begin{definition}
Let	$w_{=n}^{\ast int}(\xi)$ resp.
$w_{=n}^{\ast u}(\xi)$  be the supremum of $w$ such that
\eqref{eq:alp} has infinitely many solutions in algebraic 
integers resp. units $\alpha$ of 
degree exactly $n$. Let $w_{=n}^{int}(\xi)$ resp. $w_{=n}^{u}(\xi)$
be the supremum 
of $w$ so that \eqref{eq:hass} has infinitely many solutions in 
irreducible monic polynomials of degree exactly $n$ 
resp. irreducible monic polynomials with constant coefficient $\pm 1$
of degree exactly $n$.
\end{definition}

\begin{remark}
	The exponent $w_{=n}^{\ast u}(\xi)$ is closely related to $\tau_{n}(\xi)-1$
	for $\tau_{n}$ defined and studied in~\cite{pr}, and implicitly before in~\cite{davsh}, however we prescribe exact degree. It is not hard
	to see that
	the identity $\max_{1\leq j\leq n} w_{=j}^{\ast u}(\xi)=\tau_{n}(\xi) -1$
	holds.
\end{remark}

We start with the quantities $w_{=n}^{int}(\xi)$ and $w_{=n}^{\ast int}(\xi)$.
The generic value attained for Lebesgue almost 
all real $\xi$ is $w_{=n}^{int}(\xi)=w_{=n}^{\ast int}(\xi)=n-1$.
Similar to \eqref{eq:ds}, a variant of
Davenport and Schmidt~\cite[Lemma~1]{davsh} for exact degree that
is implictly obtained in the paper by Bugeaud and Teulie~\cite[Theoreme~5]{buteu} together with Proposition~\ref{pop} 
below shows for any real $\xi$
the chain of inequalities
\begin{equation}  \label{eq:preis}
w_{=n}^{int}(\xi) \geq
w_{=n}^{\ast int}(\xi) \ge \frac{1}{ \widehat{\lambda}_{n-1}(\xi) }, \qquad n\geq 2.
\end{equation}
A famous result of Roy~\cite{royann} verified $w_{=3}^{\ast int}(\xi)=(1+\sqrt{5})/2<2$ for a certain class of 
transcendental real numbers $\xi$. This is the minimum
value of $w_{=3}^{\ast int}(\xi)$ among transcendental real $\xi$ 
in view of \eqref{eq:preis} and \eqref{eq:2la}, and
disproves, for $n=3$, the natural conjecture that $n-1$ is a lower bound for 
the exponent for any transcendental real number $\xi$.
Our method implies some claims on the monotonicity of the exponent sequences. 

\begin{theorem} \label{t5}
 Let $1\leq m\leq n$ be integers and $\xi$ be a transcendental
 real number. Let
 \[
 \Omega_m= \Omega_{m,n}:= \min\{ w_{=m}^{int}(\xi) , \widehat{w}_{n-1}(\xi)\}.
 \]
 If
 \begin{equation} \label{eq:roest}
 \Omega_m > 2n-7,
 \end{equation}
 then
 \begin{equation} \label{eq:weem}
 w_{=n}^{int}(\xi) \geq \Omega_m.
 \end{equation}
In particular if $\Omega_{n-1} > 2n-7$ then
\begin{equation} \label{eq:tanz}
w_{=n}^{int}(\xi) \geq \min\{ w_{=n-1}^{int}(\xi) , n-1 \}.
\end{equation}
If the stronger property 
\begin{equation} \label{eq:Roest}
\Omega_m > \max\{ 2n-7, \; n-1\}
\end{equation}
holds, then additionally
\begin{equation} \label{eq:weem2}
w_{=n}^{\ast int}(\xi) \geq \frac{3}{2}\Omega_m- n + \frac{1}{2}.
\end{equation} 
\end{theorem}

The claim is most interesting again for small $n$ where the hypotheses
\eqref{eq:roest}, \eqref{eq:Roest} are mild. In fact \eqref{eq:roest} is trivially satisfied up to $n\leq 4$ when $m=n-1$, in view of \eqref{eq:preis}. On the other hand, for large $n$ it may not be feasible
to satisfy $\widehat{w}_{n-1}(\xi)>2n-7$ for any $\xi$, see the discussion on the exponents $\wo$ below Theorem~\ref{t2}.	

Now we turn towards units. 
The generic value is $w_{=n}^{u}(\xi)=w_{=n}^{\ast u}(\xi)=n-2$. 
Similar to \eqref{eq:preis}, here for any transcendental real $\xi$ we have
\begin{equation}
w_{=n}^{u}(\xi) \geq w_{=n}^{\ast u}(\xi) \geq \frac{1}{ \widehat{\lambda}_{n-2}(\xi) }, \qquad n\geq 3,
\end{equation}
as implicitly contained in~\cite{teu}. Our new result reads as follows.

\begin{theorem} \label{units}
	Let $m,n,\xi$ as in Theorem~\ref{t5} and let 
	\[
	\Omega_m^{\prime}=\Omega_{m,n}^{\prime}:= \min\{ w_{=m}^{u}(\xi) \;, \widehat{w}_{n-2}(\xi)\}.
	\]
	If
	\begin{equation*} 
	\Omega_m^{\prime} > 2n-7,
	\end{equation*}
	then
	\[
	w_{=n}^{u}(\xi)\geq \Omega_m^{\prime}.
	\]
	In particular if $\Omega_{n-1}^{\prime}>2n-7$ then
	\[
	w_{=n}^{u}(\xi) \geq \min\{ w_{=n-1}^{u}(\xi)\;, n-2 \}. 
	\]
    If the stronger property 
	\begin{equation}  \label{eq:Runitassu}
	\Omega_m^{\prime} > \max\{ 2n-7,\; n-1\}
	\end{equation}
	holds, then
		\[
	w_{=n}^{\ast u}(\xi)\geq \frac{3}{2} \Omega_m^{\prime} - n +\frac{1}{2}.
	\]
\end{theorem}

Similar remarks as for Theorem~\ref{t5} above apply.

\subsection{Irreducibility criteria for integer polynomials} \label{intro3}
Finally we derive from our method a new independent result on irreducibility of some class of polynomials. Hereby we exclude constant factors, i.e.
we refer to polynomial as reducible if it has a non-constant factor of smaller degree.
For integer polynomials
\begin{equation}  \label{eq:nopr}
P(T)=c_nT^n+\cdots+c_0,\; (c_n\neq 0), \qquad\;\;   Q(T)=d_mT^m+\cdots+d_0
\end{equation}
and $\ell, \ell_1, \ell_2\in\mathbb{Z}$ derive the polynomials
\begin{equation} \label{eq:rell}
R_{\ell}= P+\ell Q, \qquad S_{\ell}= \ell P+Q, \qquad M_{\ell_1,\ell_2}=
\ell_1 P+ \ell_2 Q.
\end{equation}
In our applications $\ell$ will mostly be a prime number, small
compared to $H(P), H(Q)$. 
For simplicity let us further define the auxiliary values
\[
\Gamma= \Gamma(P,Q,H)= \tau(c_n)\tau(d_0)\log H,
\qquad \Gamma^{\prime}= \Gamma^{\prime}(P,Q,H)= \tau(c_n)\tau(d_0)\frac{\log H}{\log \log H}
\]
with $\tau$
the number of divisors function and $H>1$ a parameter.
As usual the notation $A\ll_{u_1,\ldots,u_t} B$ means 
$A\leq cB$ for some $c>0$ that depends only on the variables $u_i$, and
if $A\ll B$ we mean the constant is absolute.

\begin{theorem}  \label{t6}
	Let $n\geq 2$ be an integer. For any $\epsilon>0$, there exists
	effectively computable
	$\delta_0=\delta_{0}(\epsilon)>0$, for which the following claims hold. 
     Let
	$P,Q$ be any integer polynomials as in \eqref{eq:nopr} 
	without common factor with
	$\deg(P)=n$ and $P(0)=c_0=0$,
	and $\deg(Q)=m<n$ and $\max\{ H(P), H(Q)\}\leq H$ for some $H>1$. 
	\begin{itemize}
		\item[(i)] 
	    Let $n\in\{ 2,3\}$. As $H\to\infty$,
		there are only $\ll \Gamma^{\prime}\ll H^{o(1)}$
		many primes
		$\ell$ for which $S_{\ell}$ defined
		above is reducible (in particular finitely many). 
		Hence, for any $\delta>0$, 
		there are $\gg_{\delta} H^{\delta}/ \log H$ many irreducible $S_{\ell}$
		for $\ell>0$ a prime up to $H^{\delta}$, and the smallest such $\ell$
		satisfies $\ell\ll H^{o(1)}$.  
		The same claims hold for $R_{\ell}$.
		\item[(ii)] Assume $n\geq 4$, $P$ has a root $\alpha$
		and $Q$ a root $\beta$ with 
		\begin{equation}  \label{eq:bau}
		|\alpha-\beta| \leq H^{-\kappa_n-\epsilon}, \qquad \kappa_n=2n-6.
		\end{equation}
		Then, as $H\to\infty$, there are only 
		$\ll_{n,\epsilon} \Gamma\ll H^{o(1)}$ many primes
		$\ell>0$ up to $H^{\delta_0}$ for which $S_{\ell}$ defined
		above is reducible. In particular, for any $0<\delta\leq \delta_0$, 
		there are $\gg_{n,\epsilon} H^{\delta}/ \log H$ many primes
		$\ell>0$ up to $H^{\delta}$ inducing irreducible $S_{\ell}$, 
		and the smallest such $\ell$
		satisfies $\ell\ll_{n,\epsilon} H^{o(1)}$.  
		The same claims hold for $R_{\ell}$.
		\item[(iii)]
	Assume $n\geq 2$ and $P$ has a root $\alpha$
	and $Q$ a root $\beta$ with
		\begin{equation}  \label{eq:baux}
	|\alpha-\beta| \leq H^{-\theta_n-\epsilon}, \qquad\qquad \theta_n=
	\begin{cases} 1, \qquad\quad\; n=2, \\
	2n-4, \quad n\geq 3.   \end{cases}
	\end{equation}
	Then, as $H\to\infty$,
	the set of coprime integer
	pairs $1\leq \ell_1,\ell_2\leq H^{\delta_0}$ that induce reducible $M_{\ell_1,\ell_2}$ has cardinality $\ll_{n,\epsilon} \Gamma\ll H^{o(1)}$. 
	In particular, for $0<\delta\leq \delta_0$,
	there are $\gg_{n,\epsilon} H^{2\delta}$ many irreducible $M_{\ell_1,\ell_2}$ for integer pairs $1\leq \ell_1< \ell_2\leq H^{\delta}$.
		\end{itemize}
\end{theorem}

\begin{remark} \label{r1}
	As pointed out to me by D. Roy, instead of $P(0)=0$ and $\deg(Q)<n$ we may alternatively assume that $P,Q$ have degree $n$ and each has a rational root (not the same). This follows by considering $P^{\ast}(T)= (c+dT)^n P(\sigma(T))$ and $Q^{\ast}(T)= (c+dT)^n Q(\sigma(T))$, where $\sigma$ 
	is a birational transformation
	$\sigma: T\mapsto (aT+b)/(cT+d)$ with any integers $a,b,c,d$ satisfying
	$ad-bc\neq 0$.  
\end{remark}

The roots $\alpha, \beta$ in (ii), (iii) may be complex.
Claim (i) follows 
directly from Theorem~\ref{t4} below obtained in~\cite{moscj}.
The special case $Q\equiv 1$ of claim (i) 
for $R_{\ell}$ has a consequence on the problem of Szegedy 
if there is an absolute bound $C=C(n)$ such that
$P+b$ is irreducible for some integer $|b|\leq C$ when $P$
is any degree $n$ integer polynomial (see also Turan's problem, for example~\cite{fil, gy, schi}). 
For $n=2$, it is remarked in~\cite{gy} that $|b|\leq 2$ suffices,
however for $n>2$ Szegedy's problem is open.
Corollary~\ref{kro} below settles a moderately growing
bound in terms of the height for cubic $P$. It has
not been noticed in~\cite{moscj} as the author was at that time
unaware of Szegedy's problem and the paper~\cite{gy}.

\begin{corollary} \label{kro}
	Assume $P(T)\in \mathbb{Z}[T]$ as in \eqref{eq:nopr} is cubic of
	height $H(P)\leq H$.
	Then there is an integer $b$ that satisfies
	$|b|\ll \tau(c_3)(\log H)^2\ll H^{o(1)}$ as $H\to\infty$ and such that $P(T)+b$ is irreducible. Moreover, the constant coefficient
	of $P(T)+b$ can be chosen a prime.
\end{corollary}

\begin{proof}
	Let $\tilde{P}=P-P(0)$, which satisfies $\tilde{P}(0)=0$ and $H(\tilde{P})\leq H(P)$. 
	By claim (i) of Theorem~\ref{t6} with $Q\equiv 1$, there are only
	$\ll \tau(c_3)\log H/ \log\log H<\tau(c_3)\log H$ many primes $\ell>0$ 
	with $\tilde{P}+\ell$ reducible. On the other hand, by Prime Number Theorem and since $\tau(c_3)\ll H^{o(1)}$ as $H\to\infty$ (see Lemma~\ref{rop} below),
	we easily check that for any $0<X\leq H$ and for $c>0$ large enough independent from $H$, the interval $[X,X+c\tau(c_3)(\log H)^2]$ 
	contains a larger number of primes.
	Application to $X=P(0)$
	shows that $\tilde{P}(T)+X+b=P(T)+b$ is irreducible for
	some integer $0<b\ll \tau(c_3)(\log H)^2\ll H^{o(1)}$,
	as $H\to\infty$. Hereby the last estimate is immediate again from Lemma~\ref{rop} below. Finally $(P+b)(0)=P(0)+b$ is prime by construction. 
	\end{proof}

For general $n$, there still exists $b$ as in the corollary
such that $P+b$ has no linear factor. We compare our bound
with a result of Gy\H{o}ry~\cite{gy}.
Using the Thue-Siegel-Roth-Schmidt
method, he~\footnote{Actually a weaker bound is proved in~\cite{gy}.
For the improved bound in \eqref{eq:jab}, Gy\H{o}ry refers to private 
correspondence with J.H. Evertse. It appears this has
never been published. } showed that in Szegedy's problem 
we can take the bound
\begin{equation}  \label{eq:jab}
|b|\leq \exp \{ (\omega+1) \log (\omega+2) (2^{17}n)^{n^3} \},
\end{equation}
where $\omega=\omega(c_n)$ denotes the number of prime divisors of
the leading coefficient $c_n$ of $P(T)$.
While \eqref{eq:jab} is strong in some cases,
in particular for monic polynomials see also~\cite{hajdu}, 
it becomes rather weak if
$c_n=H(P)$ and $c_n$ has 
many prime divisors. 
Indeed, if $c_n$ is primorial, i.e. the product of the first $N$ primes
for some $N$, then $\omega(c_n)\geq  (1- o(1))\log c_n/ \log\log(c_n)$ 
as $N\to\infty$ and the bound 
in \eqref{eq:jab} becomes a quite large power of $H$. 
Thus, if $n=3$, our bound from Corollary~\ref{kro} is considerably stronger 
than \eqref{eq:jab} in general, in terms of the height. 
For $n>3$, our condition
\eqref{eq:bau} enters and does not apply if $Q\equiv 1$, therefore
we get no contribution to Szegedy's problem.

We add a few more remarks on Theorem~\ref{t6}.
We notice that when replacing $\kappa_n$ (or $\theta_n$) by $2n-2$,
the hypothesis \eqref{eq:bau} (or \eqref{eq:baux}) can only hold for small $H\leq H_0(\epsilon)$, by Liouville's inequality~\cite[Corollary~A.2]{bugbuch} and
since $P(0)=0$ (we may exclude $\alpha=0$ for $n\geq 3$ and large $H$ since
then $|\beta-\alpha|=|\beta|\gg_{n} H^{-1}$ contradicts \eqref{eq:bau}).
Similar to~\cite{gy},
the bound for the number of reducible
polynomials in all claims depends on the factorization
of certain coefficients of $P$ or $Q$, generically
a power of $\log H$ suffices.
We should point out that Cavachi~\cite{cavachi} showed 
that $S_{\ell}$ are irreducible for every sufficiently
large prime $\ell$ (for effective versions for 
prime powers see~\cite{bb}), for every $n$ and 
without condition $P(0)=0$, 
that is for
every $P,Q$ without common factor and $\deg(Q)<\deg(P)$. For $n\in\{2,3\}$
and if $P(0)=0$, an effective bound applicable to both $R_{\ell}$ and $S_{\ell}$ is given in~\cite[Theorem~3.3]{moscj}. 
However, Theorem~\ref{t6} is concerned with small $\ell$
and the involved bounds in $H$ from~\cite{bb},~\cite{moscj} 
are by far too large for interesting applications in the direction of Theorems~\ref{H},~\ref{t1},~\ref{t2}.

The following examples, partly inspired by~\cite[Remark~2]{cavachi}, 
suggest that we do not have much freedom
regarding relaxing the conditions in Theorem~\ref{t6}.

\begin{example}  \label{exa}
	Let 
	\[
	P(T)=T^2, \qquad  Q(T)= -T^2-1,
	\]
	which satisfy the assumptions of claim (i) in Theorem~\ref{t6} 
	for $n=2$ apart from $\deg(Q)<n$, for every $H\geq 1$.
	If $\ell$ is a prime of the form 
	$\ell=N^2+1$, then
	\[
	S_{\ell}= \ell P+Q= N^2T^2-1= (NT+1)(NT-1)
	\] 
	decomposes into linear factors. If we assume that $\ell$ above is prime
	with probability $(\log \ell)^{-1}\asymp (\log N)^{-1}$, then we
	should expect $\gg H^{\delta/2}/ \log H$ reducible
	$S_{\ell}$ up to $\ell\leq H^{\delta}$ for any $\delta>0$ and $H\geq 1$.
	In particular finiteness is highly unlikely.
	Similarly, if we admit $P(0)\neq 0$, 
	then we should again expect $\gg H^{\delta/2}/ \log H$ reducible
	$R_{\ell}= P+\ell Q$ from primes $\ell>0$ up to $\ell \leq H^{\delta}$ for $P,Q$ given by
	\[
	P(T)= T^2+1, \qquad Q(T)=-1.
	\]	
	For any $n\geq 2$, take
		\[
	P(T)= T^n, \qquad Q(T)=-1.
	\]	
	They satisfy all hypotheses of claim (iii) of Theorem~\ref{t6} 
	 apart from \eqref{eq:baux}, and the claim fails as can be seen 
	by considering $\ell_1=a^n, \ell_2=b^n$ for coprime integer pairs $(a,b)$.
\end{example}

We believe that similar examples
for $S_{\ell}$ when $P(0)\neq 0$
and for $R_{\ell}$ when $\deg(Q)<n$ can be found, 
but leave this as an open problem.
On the other hand, presumably
we only require polynomials $P,Q$
without common factor for the lower bounds on 
irreducible polynomials in all claims of Theorem~\ref{t6}.
We formulate some problems.

\begin{problem}
	In context of claims (ii), (iii) of Theorem~\ref{t6},
	does $\gg H^{\delta}/ \log H$ resp. $\gg H^{2\delta}$
	for the number of irreducible $S_{\ell}$ or $R_{\ell}$ 
	resp. $M_{\ell_1, \ell_2}$ remain true without condition 
	\eqref{eq:bau} resp. \eqref{eq:baux}?
	Can we further drop the condition $\deg(Q)<n$ 
	and/or $P(0)=0$ in claims (i), (ii), (iii)?
	What if we do not restrict $\ell$ to be prime in claims (i), (ii)?
\end{problem}

\section{Auxiliary results}  \label{prep}

The following observation is implicitly
implied in Wirsing's work~\cite{wirsing}
when incorporating the refinements explained in the paragraph below~\cite[Theorem~2.7]{buschlei}.

\begin{theorem}[Wirsing; Bugeaud, Schleischitz]  \label{wbs}
	Let $n\geq 2$ be given and $P_1,P_2$ integer polynomials 
	of degree at most $n$ and without common factor. Assume
	\[
	|P_i(\xi)| \leq (\max_{i=1,2}   H(P_i) )^{-\eta}, \qquad \quad i=1,2,
	\]
	holds for some $\eta>n-1$. Then for some $i\in \{1,2\}$ the polynomial $P_i$
	has a root $\alpha$ that satisfies
	\[
	|\alpha - \xi| \ll H(P_i)^{-(\frac{3}{2}\eta-n+\frac{1}{2})-1}.
	\]
\end{theorem}

\begin{proof}
	We follow the proof of~\cite[Theorem~3.14]{bugbuch} starting below the proof
	of~\cite[Lemma~3.2]{bugbuch}, incorporating the improvement of~\cite[(3.16)]{bugbuch} obtained in~\cite{buschlei}, upon identifying
	$P_k$ with our $P_1$ and $Q_k$ with our $P_2$. In order to avoid~\cite[(3.17)]{bugbuch} from the case distinction
	in~\cite[Lemma~3.2]{bugbuch}, it suffices to have $\eta>n-1$. The remaining
	cases lead for $\alpha$ a root of $P_1$ resp. $P_2$ to the respective bounds 
	\[
	|\alpha - \xi| \ll H(P_1)^{-(\frac{3}{2}\eta-n+\frac{1}{2})-1}, \qquad
	|\alpha - \xi| \ll H(P_2)^{-(\frac{3}{2}\eta-n+\frac{1}{2})-1}
	\]
	applicable in cases~\cite[(3.18)]{bugbuch} respectively ~\cite[(3.19)]{bugbuch}, and 
	\begin{equation} \label{eq:NNe}
	|\alpha - \xi| \ll H(P_1)^{-\eta-1}, \qquad |\alpha - \xi| \ll H(P_2)^{-\eta-1}
	\end{equation}
	in the cases~\cite[(3.20)]{bugbuch} respectively~\cite[(3.21)]{bugbuch}.
	Since the hypothesis of the theorem can only hold when $\eta\le 2n-1$ by the method in~\cite{buschlei} or~\cite{davsh},
	see in particular~\cite[Lemma~3.1]{buschlei}, the weaker bound is the one in \eqref{eq:NNe} and yields the claim.
	\end{proof}

As pointed out in the proof, 
the hypothesis of the theorem implies $\eta\le 2n-1$.
Theorem~\ref{wbs} in particular shows that if its conditions hold for given $\eta>n-1$ and a sequence of pairs $P_1, P_2$ of arbitrarily large heights and all of exact degree $n$, then we have
\[
w_{=n}^{\ast}(\xi) \geq \frac{3}{2}\eta-n+\frac{1}{2}.
\]
Indeed we will show that this can be arranged for $n\leq 7$ and 
any transcendental real $\xi$ for  any value
$n<\eta<\wo$ to derive Theorem~\ref{t1}.

Recall the notation $R_{\ell}, S_{\ell}$ from \eqref{eq:rell}.
The next partial claim of~\cite[Theorem 3.3]{moscj} 
gives a criterion on $P,Q$ that guarantees that 
only for few prime values $\ell$ the polynomials $R_{\ell}, S_{\ell}$
can have a linear factor.

	\begin{theorem}[Schleischitz] \label{t4}
		Let $n\geq 2$ be an integer. 
		Let $P(T)$ with $\deg(P)=n$ and $P(0)=0$, and
		$Q(T)$ with $\deg(Q)<n$ be integer polynomials
		without common factor.
		Let $H=\max\{ H(P),H(Q)\}$.
		Then for any $\varepsilon>0$, there exists a constant
		$c=c(n,\varepsilon)>0$ not depending on $P,Q$
		such that the number of prime numbers $\ell$ for which 
		either of the polynomials $R_{\ell}$ or $S_{\ell}$ has a linear factor
		over $\mathbb{Z}[T]$, is less
		than $cH^{\varepsilon}$. More precisely, the upper bound
		$\ll_{n} \tau(c_n)\tau(d_0)\log H/ \log\log H$ holds with the notation of
		\S~\ref{intro3}. 
\end{theorem}

For the last claim see the comments below~\cite[Remark~1]{moscj} and
also observe that within the proof of~\cite[Theorem 3.3]{moscj} we 
can win another double logarithm when estimating $\omega(N_s)\ll \log |N_s|/\log\log |N_s|$ in place of $\omega(N_s)\ll \log |N_s|$
which holds true as well, see Lemma~\ref{rop} below.  
A problematic issue is that Theorem~\ref{t4} does not apply to $Q$ 
of degree precisely $n$, which complicates the proofs of our main
results. The assumption $P(0)=0$ in Theorem~\ref{t4}
is slightly disturbing as well.
However, Example~\ref{exa} above 
demonstrates that we have to be very careful with generalizations.

The following well-known estimates for the height of products
are often referred to as Gelfond's Lemma.
It can be found for example in~\cite{wirsing}.

\begin{lemma}[Gelfond]  \label{gel}
	Let $d\geq 1$ be an integer. For polynomials $P_1, P_2$ 
	of degree at most $d$ the heights
	are multiplicative up to a factor, that is there is $c(d)>0$ so that 
	\[
	c(d)^{-1} H(P)H(Q) \leq H(PQ) \leq c(d) H(P)H(Q).
	\]
\end{lemma}

An immediate consequence was observed by Wirsing~\cite[Hilfssatz 1]{wirsing}.

\begin{lemma}[Wirsing]  \label{hiw}
	Assume $d\geq 1$ is an integer and $\xi$ is any real number, 
	and $P$ is a non-zero integer polynomial 
	of degree at most $d$. Suppose that for some $\eta$ the estimate 
	\[
	|P(\xi)|\leq H(P)^{-\eta}
	\]
	holds. Then there is an irreducible divisor $W$ of $P$ (possibly equal to $P$) satisfying
	\[
	|W(\xi)|\ll_{d,\xi} H(W)^{-\eta}.
	\]
\end{lemma}

\begin{remark}  \label{lamm}
	If $H(P)\to\infty$
	and $\xi$ is transcendental, then for the induced irreducible
	$W$ we may assume $H(W)\to\infty$ as well. Indeed, if otherwise
	$H(W)\le X_0$ is bounded for a sequence of $P$ as above, then 
	we have a uniform lower bound $|W(\xi)|\gg_{X_0,\xi,d} 1$. But then
	we may consider $P/R$ instead of $P$, which has degree less than $d$ and height $\gg_{X_0,d} H(P)$ by Gelfond's Lemma~\ref{gel}. Since $d$
	is fixed, repeating the argument, we must arrive at some
	irreducible divisor of $P$ with large height and the 
	stated property.
	\end{remark} 

The lemma is not the precise formulation by Wirsing, but
from the proof Wirsing provided, the claim of Lemma~\ref{hiw} 
is evident.
We will frequently apply a direct consequence of~\cite[Lemma~3.1]{buschlei} that
can be considered a variant of Liouville's inequality~\cite[Corollary~A.2]{bugbuch}.

\begin{lemma}[Bugeaud, Schleischitz]  \label{sb}
	Let  $U_{1},U_{2}$ be integer polynomials of
	degrees $d_{1}>0$ and $d_{2}>0$ respectively and without common
	non-constant factor over $\mathbb{Z}[T]$. 
	If we let $H=\max_{i=1,2} H(U_i)$, then for any real $\xi$ we have
	\[
	\max_{i=1,2} |U_{i}(\xi)| \geq  c H^{-d_{1}-d_{2}+1},
	\]
	for some constant $c=c(d_{1}, d_{2}, \xi)>0$ that does not depend
	on the $U_i$.
\end{lemma}

The claim is also true for complex $\xi$ with the same proof.
The next crucial lemma stems from Lemma~\ref{sb} by a pigeon hole
principle. Although not particularly deep, it
appears to be new.

\begin{lemma} \label{lemur}
	Let $d\geq 1$ be an integer, $\xi$ be a real number 
	and $\mu>2d-1$ a real number. Then
	for every $H>1$, at most $\ll_{\mu} \log H$ pairwise coprime 
	integer polynomials $Q$
	of degree at most $d$ satisfy the estimates
	\begin{equation}  \label{eq:hei}
	H(Q)\le H, \qquad |Q(\xi)| \leq H(Q)^{-\mu}.
	\end{equation}
\end{lemma}

The lemma applies in particular to pairwise distinct irreducible polynomials.

\begin{proof}[Proof of Lemma~\ref{lemur}]
	Let 
	\begin{equation}  \label{eq:null}
	\epsilon= \frac{\frac{\mu}{2d-1}-1}{2}>0, \qquad\qquad \theta=\frac{\mu}{1+\epsilon}>2d-1.
	\end{equation}
	Then by Lemma~\ref{sb} for $d_1=d_2=d$, for large $H\geq H_0=H_{0}(d,\theta)$
	the properties
	\[
	H=\max_{i=1,2} H(U_i),\qquad  |U_i(\xi)| \leq H^{-\theta},\qquad\qquad i\in\{1,2\},
	\] 
	cannot both hold for two distinct polynomials $U_1, U_2$
	of degree at most $d$ and without common factor. 
	Thus these pairs of counterexamples are contained in $\mathcal{T}$ defined as the finite set
	of integer polynomials of degree at most $d$ and of height at most $H_0$.
	
	Let $Q_1, Q_2, \ldots,Q_{h}$ for some $h>0$ be a collection of polynomials
	as in the lemma satisfying \eqref{eq:hei}, 
	ordered by increasing heights (and arbitrary labelling when heights coincide).
	First assume there are $Q_v, Q_{v+1}$ with $v$
	large enough that $H(Q_v)\geq H_{0}$ and
	the property 
	\begin{equation}  \label{eq:v}
	1 \leq \frac{\log H(Q_{v+1})}{\log H(Q_v)} \leq 1+\epsilon.
	\end{equation}
	Then 
	for $H:= H(Q_{v+1})$, by combining \eqref{eq:hei}, \eqref{eq:null} and \eqref{eq:v} both $Q_v, Q_{v+1}$ satisfy
	\[
	H(Q_i)\leq H, \qquad  |Q_i(\xi)|\leq H^{-\theta} , \qquad\qquad i\in \{v,v+1\}.
	\]
	This
	contradicts our claim above. Hence we cannot have 
	\eqref{eq:v} for any large $v$, thus
	\[
	\frac{\log H(Q_{v+1})}{\log H(Q_v)} > 1+\epsilon, \qquad v\geq v_0.
	\]
	But this means that up to a given height $H$ we can have
	at most $\sharp \mathcal{T}+ \log_{1+\epsilon} (H/H_0)\ll \log H$ 
	many polynomials as in the lemma, where $\sharp \mathcal{T}$ denotes the cardinality of $ \mathcal{T}$
	and the logarithm notation means taking 
	the logarithm to base $1+\epsilon$. 
	The proof of the lemma is complete.
\end{proof}

It would be desirable to establish non-trivial bounds
for $\mu$ a bit smaller than in the lemma.
For convenience of the reader 
we next provide a short proof of some well-known fact that 
will be applied occasionally.

\begin{proposition} \label{pop}
	If $\alpha$ is a root of a polynomial $P$ of degree $n$ and height 
	$H(P)=H$, then if $|\xi-\alpha|\leq 1$ we have
	\[
	|P(\xi)|\ll_{n,\xi} H |\xi-\alpha|.
	\]
\end{proposition}

\begin{proof}
	Since $P(\alpha)=0$ by intermediate value theorem of differentiation
	\[
	|P(\xi)|= |P(\xi)-P(\alpha)|= |\xi-\alpha|\cdot  |P^{\prime}(\eta)|
	\]
	for some $\eta$ between $\xi$ and $\alpha$. Hence $|\eta|\leq |\xi|+1$
	 and as $P^{\prime}$ has height $H(P^{\prime})\leq nH(P)$ we estimate
	\[
	|P^{\prime}(\eta)|\leq (n+1)\max\{ 1,|\eta|^n\}\cdot H(P^{\prime})
	\ll_{n,\xi} H
	\]
	and the claim follows.
	\end{proof}

Again the proof works for $\alpha, \xi\in\mathbb{C}$ as well.
Finally we require two estimates from analytic number theory
that have already been quoted above at some places.

\begin{lemma}  \label{rop}
	Let $\epsilon>0$. The number of divisors $\tau(N)$ of an integer $N\neq 0$ is $\tau(N)\ll_{\epsilon} |N|^{\epsilon}$. 
	The number of its prime divisors satisfies $\omega(N)\leq (1+\epsilon)\log N/\log\log N$ for $N\geq N_{0}(\epsilon)$. If $N$
	is primorial, then $\omega(N)\geq (1-\epsilon)\log N/\log\log N$ for $N\geq N_{0}(\epsilon)$.
	\end{lemma}

 The first estimate can be found for example in the book of Apostol~\cite[page~296]{apostol}, the claims on $\omega$ can 
be obtained from the Prime Number Theorem and are well-known.

\section{A large set of polynomials small at $\xi$}    \label{small}

We will assume $n\geq 3$ as otherwise the claims Theorem~\ref{t1},~\ref{t2} follow from \eqref{eq:jm} and \eqref{eq:stand}. 
Moreover, we can assume
\[
\wo > n
\]
since otherwise if $\wo=n$ by \eqref{eq:w} and \eqref{eq:ds} 
we get $\ws\geq \wo=n$, a
stronger claim than Theorem~\ref{H} and Theorem~\ref{t1}, and by \eqref{eq:stand} 
we derive Theorem~\ref{t2} as well.

Let $\epsilon>0$.
By Lemma~\ref{hiw} and the definition of the exponent $w_{n}(\xi)$,
there exist {\em irreducible} polynomials
$P$ of arbitrarily large height (but maybe of degree
less than $n$) and
\begin{equation} \label{eq:ipo}
|P(\xi)| \leq H(P)^{-\w + \epsilon }.
\end{equation}
In fact this is the original formulation of~\cite[Hilfssatz~1]{wirsing}.
We stress here that in fact for our method below we 
only require the weaker estimate
\begin{equation}  \label{eq:enough}
|P(\xi)| \leq H(P)^{-\wo + \epsilon },
\end{equation}
for some irreducible $P$.
This will enable us to transition to different $P$ if needed below.

Write $P_k$ for a sequence of $P$ with the above properties and $H(P_k)\to\infty$, and
let for simplicity $H_{k}= H(P_{k})$ denote the height of $P_{k}$. 
Let $c\in (0,1)$ be small enough that every 
integer polynomial of degree at most $n$ and 
height at most $cH_{k}$ is not a multiple of $P_{k}$,
which can be done by Gelfond's Lemma~\ref{gel} 
(this argument already goes back
to Wirsing~\cite{wirsing} as well). By definition of $\wo$,
for large $k$ there is a non-zero integer 
polynomial $Q_{k}$ of degree at most $n$ so that 
\begin{equation}  \label{eq:duenn}
H(Q_k) \leq cH_{k}, \qquad |Q_{k}(\xi)| < H_{k}^{ -\wo +\epsilon }.
\end{equation}
We may assume
\begin{equation}  \label{eq:wemay}
T\nmid Q_k
\end{equation}
as otherwise if $Q_k(T)= T^{h_k}Q_{k}^{\ast}(T)$
for some $h_k\in \mathbb{Z}$ and $Q_{k}^{\ast}\in \mathbb{Z}[T]$ with $Q_{k}^{\ast}(0)\neq 0$,
then we take instead $Q_k^{\ast}$ which shares the properties we want (coprime to $P_k$, same height $H(Q_{k}^{\ast})= H(Q_k)$, and $|Q_k^{\ast}(\xi)|\asymp_{n,\xi} |Q_k(\xi)|$) so
the property \eqref{eq:duenn} is preserved.
Since $P_{k}$ is irreducible the polynomials $P_{k}, Q_{k}$ have no
common factor. Let $u_{k}$ be the degree of $P_{k}$, and $f_{k}$ the
degree of $Q_k$. Let
\[
\delta\in (0,\; \min\{ 1/3,\; \wo-n\})
\]
arbitrary but fixed.
Now for every $k\geq 1$
we consider the set $\mathcal{S}_{k}=\mathcal{S}_{k}(\delta)$ 
consisting of the integer polynomials $S_{\ell}(T)= S_{k,\ell}(T)$ in
variable $T$ defined by
\[
S_{\ell}= \ell T^{n-u_{k}}P_{k} + Q_{k}, \qquad\qquad 1\leq \ell \leq H_{k}^{\delta},\; \ell \; \text{prime}.
\]
Up to at most two 
exceptional values of $\ell$ for each $k$,
all $S_{\ell}$ have degree precisely $n$
and height at least $H_k/2$ (each condition 
having at most one exceptional value).
Removing these two $\ell$ if necessary,
write $\mathscr{L}=\mathscr{L}(k)$ for the remaining 
set of primes as above, and let 
$\mathcal{S}_{k}=\{ S_{k,\ell}\!\!: \ell\!\in\! \mathscr{L}\}$. Then
\begin{equation}  \label{eq:001a}
\min_{\ell\in \mathscr{L}(k) } H(S_{k,\ell}) \ge H_{k}/2.
\end{equation}
Consequently, we may choose $H_k$ growing fast enough that the
sets $\mathcal{S}_{k}$ are disjoint for distinct $k$.
We occasionally omit the dependence of $k$ in $S_{\ell}$ 
in notation for readability. Let us also directly 
define the twisted polynomials
\[
R_{\ell}=R_{k,\ell}=T^{n-u_k}P_k+\ell Q_k, \qquad\qquad 1\leq \ell \leq H_{k}^{\delta},\; \ell \; \text{prime},
\] 
that will be required for the proofs 
of Theorems~\ref{t5}, \ref{units}, \ref{t6}. 
We notice that again for each $k$ and at most one index $\ell$ the polynomial 
$R_{k,\ell}$ has degree less than $n$ (if so then $\deg(Q_{k})=n$), and 
moreover for all but at most one exceptional index $\ell$ 
we get $H(R_{\ell})>H_k^{1/2}/2$. Indeed, since
$\ell\le H_k^{\delta}$, $\delta<1/3$, is small compared to $H_k$, for large $k$ we certainly require $H(Q_k)\ge H_k^{1/2}$ for an index $\ell_0$ to exist such that
$H(R_{\ell})< H_k^{1/2}/2$. But if so, then for any other index $\ell\ne \ell_0$
we have $H(R_{\ell})> H_k^{1/2}/2$. Denote by $\mathscr{R}=\mathscr{R}(k)$ the remaining
set of primes excluding at most two indices where both conditions are true and let
$\mathcal{R}_k=\{ R_{k,\ell}: \ell\in \mathscr{R}(k) \}$.
So by construction
\begin{equation}  \label{eq:001b}
\min_{\ell\in \mathscr{R}(k) } H(R_{k,\ell})\ge H_k^{1/2}/2.
\end{equation}
An obvious conseuqence of \eqref{eq:001a}, \eqref{eq:001b} that will mostly
be sufficient is
\begin{equation} \label{eq:non}
\lim_{k\to\infty} \min_{\ell\in \mathscr{L}(k) } H(S_{k,\ell})=\infty, \qquad \lim_{k\to\infty} \min_{\ell\in \mathscr{R}(k) } H(R_{k,\ell})=\infty.
\end{equation}
Below we will mostly restrict to presenting results for $S_{\ell}$ only but it is
understood they also hold for $R_{\ell}$ with a very similar proof. Only 
if there are notable twists in the proofs for $R_{\ell}$ we
will separately comment on the modifications to be done.
Since $\delta<1$,
by Prime Number Theorem
the cardinality of the set $\mathcal{S}_{k}$ is 
$\gg H_{k}^{\delta}/ \log H_{k}$ with an absolute implied
constant.
Moreover for every $\ell\in \mathscr{L}$ 
we see
\begin{equation}  \label{eq:hirte}
H(S_{\ell}) \leq \ell H(P_k) + H(Q_k) \leq 2H_{k}^{1+\delta}
\end{equation}
and from \eqref{eq:enough}, \eqref{eq:duenn} moreover
\begin{equation}  \label{eq:kuh}
|S_{\ell}(\xi)| \leq |\ell|\cdot |P_{k}(\xi)| + |Q_{k}(\xi)| \leq
(\ell+1)H_{k}^{ -\wo +\epsilon }
\ll H_{k}^{ -\wo +\delta+\epsilon }.
\end{equation}
Notice that indeed the bound in
\eqref{eq:enough} in place of \eqref{eq:ipo}
on $P(\xi)$ is sufficient for these claims.
Combining the last two estimates \eqref{eq:hirte}, \eqref{eq:kuh}, we see in particular
\begin{equation} \label{eq:hase}
|S_{\ell}(\xi)| \ll H(S_{\ell})^{-\wo+\delta^{\prime}},\qquad
\delta^{\prime}>0,
\end{equation}
where $\delta^{\prime}$ is arbitrarily small for small enough
$\delta, \epsilon$, 
and independent from $\ell\in \mathscr{L}$.
In view of Theorem~\ref{wbs},
the key point is to show that at least two of the $S_{\ell}, \ell\in\mathscr{L}$, 
are irreducible. Indeed we prove there are many such polynomials.
To do this, we use our information
on $P_k, Q_k$ and distinguish several situations.
First observe that at least one of the four following cases must occur for 
infinitely many $k$:

\begin{itemize}
    \item Case 1: $u_k<n, f_k<n$.
	\item Case 2: $u_k<n, f_k=n$.
	\item Case 3: $u_k=f_k=n$.
	\item Case 4: $u_k=n, f_k<n$.
\end{itemize}

We treat each case separately. Before we do so, we state
a few general observations.
Fix $k$. Let $\textbf{I}=\textbf{I}(k)\subseteq \mathscr{L}(k)$ denote the set of indices $\ell$
for which $S_{\ell}\in \mathcal{S}_{k}$ is reducible, meaning
it has a proper factor, i.e. not constant or with constant cofactor. 
A constant factor does not matter as we can consider $S_{\ell}/d\in\mathbb{Z}[T]$
if $d$ divides all coefficients of $S_{\ell}$, which still has 
the desired properties \eqref{eq:hirte}, \eqref{eq:kuh}.
We want to show that the set $\textbf{I}$ is small. 
%
%
It follows from Wirsing's Lemma~\ref{hiw} 
and \eqref{eq:hase} that for $\ell\in \textbf{I}$
we can find a factorization
\[
S_{\ell} = A_{\ell} B_{\ell},\qquad  \text{resp.} \;\; R_{\ell} = A_{\ell} B_{\ell} 
\] 
where $A_{\ell}$ is irreducible of degree at most $n-1$ and satisfies
\begin{equation} \label{eq:horse}
|A_{\ell}(\xi)|
\ll_{n,\xi} H(A_{\ell})^{ -\wo + 2\delta^{\prime} },
\end{equation}
with $\delta^{\prime}$ as above. Anywhere in the paper
$A_{\ell}$ will be implicitly considered irreducible.
We may assume $H(A_{\ell})\to \infty$ as $k\to\infty$ for any $\ell=\ell(k)\in\mathscr{L}(k)$ resp. $\ell=\ell(k)\in\mathscr{R}(k)$
under consideration by definition of $\mathscr{L}(k)$, $\mathscr{R}(k)$ and \eqref{eq:non} and Remark~\ref{lamm}. 
The same argument apply for $R_{\ell}$.
Moreover, keep in mind for the sequel the following fact.

\begin{proposition}  \label{aktiv}
	Two distinct polynomials
	in $\mathcal{S}_{k}$ resp. in $\mathcal{R}_{k}$ have no common non-constant polynomial divisors.
\end{proposition}

\begin{proof}
		Any common divisor of some $S_{\ell_1}$ and $S_{\ell_2}$
	for $\ell_1\neq \ell_2$ would have to divide their difference 
	$S_{\ell_1}-S_{\ell_2}$ which is a 
	non-zero scalar multiple of $T^{n-u_k}P_{k}$. 
	But clearly $P_{k}$ is coprime to all $S_{\ell}$ since $TP_{k}$
	and $Q_{k}$ are coprime by \eqref{eq:wemay}, so the divisor must be a constant.
    Similarly, $R_{\ell_1}-R_{\ell_2}$ is a scalar multiple of $Q_k$ which
    is clearly coprime to any $R_{\ell}$ again since $TP_{k}$
    and $Q_{k}$ are coprime, hence any joint divisor
    of $R_{\ell_1}, R_{\ell_2}$ must be constant.
\end{proof}

It turns out the proofs of the cases $n=6, n=7$, especially $n=7$, 
are more tidious, so we decide to treat them separately.

\section{Proof of Theorems~\ref{t1},~\ref{t2} for $n\leq 5$  } \label{n5}
 
 We treat the four cases from the last section separately.
 First we prove Theorem~\ref{t1} in each case, and later explain
 how to derive Theorem~\ref{t2}.

\subsection{Case 1} \label{case1}

As indicated, the main step is to show an irreducibility result as follows.

\begin{theorem}  \label{t3}
	For $n\leq 5$, the set $\textbf{I}$
	defined above has cardinality $\sharp \textbf{I}\ll H_k^{o(1)}$ 
	as $k\to\infty$.
	Hence there are $\gg H_{k}^{\delta}/\log H_{k}$
	irreducible polynomials in $\mathcal{S}_{k}$ resp. $\mathcal{R}_{k}$, each of them 
	of degree precisely $n$. 
\end{theorem}

Fix large $k$.
We split the set $\textbf{I}$ into two sets
\[
\textbf{I}= \textbf{J}^{(1)} \cup \textbf{J}^{(2)}
\]
where for $A_{\ell}$ as above the subsets are given by
\[
 \textbf{J}^{(1)}= \{ \ell \in  \textbf{I}: 1\leq \deg(A_{\ell})\leq n-2 \}, \qquad  \textbf{J}^{(2)}= \{ \ell \in  \textbf{I}: \deg(A_{\ell})= n-1 \}.
\]
The irreducible polynomial divisor $A_{\ell}$ of $S_{\ell}$ defined above may not be unique, however this is not a problem as we may allow that the union is not disjoint.
We first show that $\textbf{J}^{(1)}$ is small, 
more precisely has cardinality $\sharp \textbf{J}^{(1)}\ll \log H_k$.
Notice Gelfond's Lemma~\ref{gel} and \eqref{eq:hirte} imply
\begin{equation} \label{eq:geller}
H(A_{\ell}) \leq H(A_{\ell})H(B_{\ell})
\ll_{n} H(A_{\ell}B_{\ell}) = H(S_{\ell})\leq 2H_{k}^{1+\delta}.
\end{equation}
Recall
we can assume strict inequality $\wo>n$.
Hence we may assume that $\delta$ and $\epsilon$ and thus 
$\delta^{\prime}$ are small enough so that the exponent in 
\eqref{eq:horse} is still strictly 
smaller than $-n$, i.e
\begin{equation}  \label{eq:dd}
|A_{\ell}(\xi)|
\ll H(A_{\ell})^{ -\theta }, \qquad \theta>n.
\end{equation}
Now, on the other hand, for $n\leq 5$ notice that
$2\deg(A_{\ell})-1\leq 2(n-2)-1\leq n<\theta$.  
Thus in view of Lemma~\ref{lemur} with $d= \deg(A_{\ell})$,
only $\ll \log H_k$ many distinct polynomials $A_{\ell}$ may occur. 
Since the $A_{\ell}$ are pairwise distinct for different $\ell$ by Proposition~\ref{aktiv}, we conclude that only for $\ll \log H_k$ 
many $\ell\in \textbf{J}^{(1)}$ we may have
\eqref{eq:dd}. In other words $\sharp \textbf{J}^{(1)}\ll \log H_k$ as claimed. Notice we did not use the condition of case 1 here.

Now assume $\ell\in \textbf{J}^{(2)}$, i.e.
\begin{equation} \label{eq:n-1}
\deg (A_{\ell})=n-1, \qquad \deg(B_{\ell})=1.
\end{equation}
Then $S_{\ell}$ has a linear factor. But in our case 1 we can apply
Theorem~\ref{t4} to $P(T)=T^{n-u_{k}}P_{k}$
and $Q=Q_{k}$, which tells us that there is indeed only a small 
number $\sharp \textbf{J}^{(2)}\ll H_{k}^{o(1)}$ of these $\ell$ as well.
We point out that
the latter argument for estimating the cardinality
of $\textbf{J}^{(2)}$ does not require any 
restriction on $n$. 
Hence a total of
\[
\sharp(\mathscr{L}\setminus \textbf{I}) \geq
\sharp \mathscr{L}- \sharp \textbf{J}^{(1)}- \sharp \textbf{J}^{(2)} \gg
H_k^{\delta}/\log H_k-\log H_k- H_k^{o(1)}\gg H_k^{\delta}/\log H_k
\] 
indices $\ell\in \mathscr{L}\setminus \textbf{I}$ must
remain where $S_{\ell}$ is irreducible.
The proof of Theorem~\ref{t3} in case 1 is completed.

The completion of the proof of Theorem~\ref{t1} in case 1 is done
via Theorem~\ref{wbs}. We take any two irreducible polynomials $S_{\ell_{1}}, S_{\ell_{2}}$ from Theorem~\ref{t3}, i.e.
with $\ell_i \in \mathscr{L}\setminus \textbf{I}$,
which are obviously coprime and satisfy
\eqref{eq:hirte} and \eqref{eq:kuh} for arbitrarily small 
positive $\delta, \epsilon$. By Theorem~\ref{wbs} we get a root $\alpha$ of either $S_{\ell_{1}}$ or $S_{\ell_{2}}$ that satisfies
\[
|\alpha- \xi| \ll H(\alpha)^{-(\frac{3}{2}\wo -n +\frac{1}{2}) -1 +\varepsilon},
\]
for arbitrarily small $\varepsilon>0$ if $\delta, \epsilon$ are 
chosen sufficiently small.
Since $S_{\ell_{1}}$ and $S_{\ell_{2}}$ both are irreducible of degree
exactly $n$, the claim follows.

\subsection{Case 2}  \label{case2}
Again we prove a variant of Theorem~\ref{t3}. We again consider the set $\mathcal{S}_k$ and $\textbf{I}$.
The estimate $\textbf{J}^{(1)}\ll \log H_k$ works precisely as in
case 1. We have to estimate $\textbf{J}^{(2)}$.
So assume \eqref{eq:n-1} holds for
$S_{\ell}\in \mathcal{S}_k$ with $\ell\in \textbf{I}$. 
Write $B_{\ell}(T) = q_{\ell}T-p_{\ell}$ for each $\ell\in \textbf{I}$. We may assume $p_{\ell}, q_{\ell}$ are coprime, otherwise 
we consider $S_{\ell}/(p_{\ell},q_{\ell})\in \mathbb{Z}[T]$ in place of $S_{\ell}$, which has both smaller height and evaluation at $\xi$,
and the same argument below works. 
A problem that arises
is that here we cannot directly apply Theorem~\ref{t4} since
$Q_k$ has degree $n$. However, with some effort
we can still reduce our problem to case 1.
We first show the following claim.

\textbf{Claim:}
There are at most $\ll_{\xi} H_k^{o(1)}$ many indices $\ell \in \textbf{J}^{(2)}$
for which we have
\begin{equation}  \label{eq:assuan}
|B_{\ell}(\xi)| \leq 1.
\end{equation}

We prove the claim. 
We consider $k$ fixed and may write
\begin{equation} \label{eq:heute}
T^{n-u_k}P_{k}(T)= c_{1}T+c_{2}T^{2}+\cdots+c_{n}T^{n}, \qquad 
Q_k(T)=d_{0}+d_{1}T+\cdots +d_{n}T^{n},
\end{equation}
with $d_{0}, c_{n}$ non-zero. Indeed, \eqref{eq:wemay} implies $d_0\ne 0$,
and $c_{n}\neq 0$ since $P_k$ has 
exact degree $u_k$ by assumption. Notice $c_0=0$ since $n-u_k>0$ in case 2, 
and possibly some other $c_j$ vanish as well. We may assume $d_0>0$ as
$-Q_k$ shares the same properties with $Q_k$.

Now the identity $S_{\ell}(p_{\ell}/q_{\ell})= B_{\ell}(p_{\ell}/q_{\ell}) =0$ 
after multiplication with $q_{\ell}^n\neq 0$ can be written
in coefficients as
\begin{equation} \label{eq:det}
\ell(c_{1}p_{\ell}q_{\ell}^{n-1}+c_{2}p_{\ell}^{2}q_{\ell}^{n-2}+\cdots+c_{n}p_{\ell}^{n})
+d_{0}q_{\ell}^{n}+d_{1}p_{\ell}q_{\ell}^{n-1}+\cdots+d_{n}p_{\ell}^{n}=0.
\end{equation}
Reducing modulo $p_{\ell}$ we see $p_{\ell}| d_0$. Hence, 
among all $\ell\in \textbf{I}$, since $|d_0|\leq H(Q_k)<cH_k<H_k$
and by Lemma~\ref{rop}
at most $2\tau(d_0)\ll H_k^{o(1)}$ distinct integers $p_{\ell}$ can occur,
where $\tau$ denotes the number of positive divisors of an integer.
However, by assumption $|B_{\ell}(\xi)|=|q_{\ell}\xi-p_{\ell}|\leq 1$ and as $\xi$ is fixed, we see that 
$q_{\ell}$ can only take $\ll_{\xi} 1$ different values for given $p_{\ell}$.
Hence in total we only get $\ll_{\xi} H_k^{o(1)}$ pairs $(p_\ell, q_\ell)$
or equivalently $\ll_{\xi} H_k^{o(1)}$ many distinct $B_{\ell}$ among
all $\ell\in \textbf{J}^{(2)}$.
Now it cannot happen that two distinct $\ell\in \textbf{J}^{(2)}$ induce
the same $B_{\ell}$ in view of Proposition~\ref{aktiv} and
since any $B_{\ell}$ divides $S_{\ell}$. 
Thus indeed the assumption \eqref{eq:assuan}
can only hold for $\ll_{\xi} H_k^{o(1)}$
many indices within $\textbf{J}^{(2)}$, and the claim is proved.

	For $R_{\ell}$ instead, in place of \eqref{eq:det} by $R_{\ell}(p_{\ell}/q_{\ell})= B_{\ell}(p_{\ell}/q_{\ell}) =0$  we have
	\[
	c_{1}p_{\ell}q_{\ell}^{n-1}+c_{2}p_{\ell}^{2}q_{\ell}^{n-2}+\cdots+c_{n}p_{\ell}^{n}
	+\ell(d_{0}q_{\ell}^{n}+d_{1}p_{\ell}q_{\ell}^{n-1}+\cdots+d_{n}p_{\ell}^{n})=0.
	\]
	Reducing modulo $q_n$ we see $q_{\ell}|c_n$. Hence again  
	only $2\tau(c_n)\ll H_k^{o(1)}$ many distinct $q_{\ell}$ can appear, and very similarly
	we again conclude that only few $B_{\ell}$ satisfying \eqref{eq:assuan} 
	are induced.

Now recall $\sharp \mathscr{L}\gg H_k^{\delta}/ \log H_k$. If the
index difference set $\mathscr{L} \setminus \textbf{I}$ has cardinality at
least two, then we pick any two indices in this set and apply the concluding argument of \S~\ref{case1} based on Theorem~\ref{wbs} and are done. Hence, in 
view of the claim above, we may assume
there are at least $\sharp \textbf{I}- H_k^{o(1)}\geq \sharp \mathscr{L}-1-\log H_k-H_{k}^{o(1)}\gg H_k^{\delta}/\log H_k$ 
many $\ell_0\in \textbf{J}^{(2)}$ for which we have
\[
|B_{\ell_0}(\xi)|> 1.
\]
Take any such index $\ell_0$. 
In view of \eqref{eq:kuh} we get 
\[
|A_{\ell_0}(\xi)| = \frac{ |S_{\ell_0}(\xi)| }{|B_{\ell_0}(\xi)|}
< |S_{\ell_0}(\xi)|\ll H_{k}^{ -\wo +\delta+\epsilon },
\]
and $\delta, \epsilon$ can be arbitrarily small. Hereby we should notice
that again \eqref{eq:enough} suffices for these conclusions,
i.e. the last inequality above, as in case 1.
But $A_{\ell_0}$ has degree $n-1<n$ and by Gelfond's Lemma~\ref{gel}
is has height $H(A_{\ell_0})\ll H(S_{\ell_0})\leq H_{k}^{1+\delta}$. 
Moreover, it is clearly also coprime to $P_{k}$ by Proposition~\ref{aktiv}. 
Thus we can use $A_{\ell_0}$ instead of $Q_{k}$ in the argument of case 1, 
that is we consider the set $\widetilde{\mathcal{S}}_k$ consisting of the polynomials 
\[
\widetilde{S}_{\ell}(T) = \ell T^{n-u_k} P_{k} + A_{\ell_0}, \qquad\qquad 1\leq \ell \leq H_{k}^{\delta},\; \ell \; \text{prime}. 
\]
We prove Theorem~\ref{t3} for $\widetilde{\mathcal{S}}_k$.
Again these integer polynomials all have exact degree $n$, satisfy
\begin{equation}  \label{eq:A}
H(\widetilde{S}_{\ell}) \ll H_{k}^{1+\delta}
\end{equation}
and
\begin{equation}  \label{eq:B}
|\widetilde{S}_{\ell}(\xi)| \ll \ell\cdot |P_{k}(\xi)| + |A_{\ell_0}(\xi)| 
\ll H_{k}^{ -\wo +\delta+\epsilon },
\end{equation}
and proceeding as in case 1, whose hypothesis applies, 
we see that for many $\ell\in\mathscr{L}$ 
the polynomials $\widetilde{S}_{\ell}$ are irreducible. 
For the accordingly defined polynomials $\widetilde{R}_{\ell}$ 
the exponent in \eqref{eq:B} 
is $-\wo +2\delta+\epsilon$ instead, however
this is not critical for our argument below.
Thus we conclude by Theorem~\ref{wbs} as in case 1 again.

\subsection{Case 3} 
Here again we cannot use Theorem~\ref{t4} immediately
since $P(0)=P_{k}(0)\neq 0$. Assume first $Q_k$ is irreducible
for infinitely many $k$. 
Then by Theorem~\ref{wbs}
either $P_k$ or $Q_k$ has a root
$\alpha$ satisfying
\[
|\alpha- \xi| \ll H(\alpha)^{-(\frac{3}{2}\wo -n +\frac{1}{2}) -1 +\varepsilon},
\]
for arbitrarily small $\varepsilon>0$.
Since $P_k$ and $Q_{k}$ are now both irreducible of degree
exactly $n$ and $H_k\to \infty$ by assumption and clearly $H(Q_k)\to\infty$ follows from \eqref{eq:duenn}
and the transcendence of $\xi$, we are done. 

Assume otherwise $Q_k$ is reducible for all large enough $k$.
Then by Wirsing's Lemma~\ref{hiw}
it has a factor $W_k$ of degree less than $n$ and approximation quality
$\wo$, i.e. \eqref{eq:enough} holds for $P=W_k$.
By Remark~\ref{lamm} we may assume $H(W_k)\to\infty$ as $k\to\infty$.
Hence we replace $P_k$ by $W_k$ and accordingly redefining $Q_k$ 
we find ourselves in the situation of cases 1 or 2, for all large $k$.
Recalling we have seen that \eqref{eq:enough} suffices for the claim
in cases 1 and 2, this case is done as well.

\subsection{Case 4} 

This can be reduced to cases 1 or 2 again. If $f_k<n$ for some $k$
then by Wirsing's Lemma~\ref{hiw} the polynomial $Q_k$
has an irreducible factor $W_k$
(possibly equal to $Q_k$)
of degree smaller than $n$ and approximation quality $\wo$, i.e. 
\eqref{eq:enough} holds for $P=W_k$. Moreover again by Remark~\ref{lamm} we may assume $H(W_k)\to\infty$.
We again can start with $W_k$ instead of $P_k$ and land in cases 1 or 2, 
again since \eqref{eq:enough} is sufficient in these cases.

The proof of Theorem~\ref{t1} is complete for $n\leq 5$. 
We explain how Theorem~\ref{t2} follows from the above exposition.

\subsection{Proof of Theorem~\ref{t2} for $n\leq 5$}

In cases 3, 4 of the proof of Theorem~\ref{t1} the claim is trivial
as we can just take the polynomials $P_k$.
In cases 1,2 when $P_k$ has degree less than $n$,
we have shown above that for every large $k$, 
many polynomials $S_{\ell}\in \mathcal{S}_{k}$ (resp. $\widetilde{S}_{\ell}\in \widetilde{\mathcal{S}}_k$)
are irreducible of degree precisely $n$ 
and satisfy \eqref{eq:hirte} and \eqref{eq:kuh} 
(resp. \eqref{eq:A} and \eqref{eq:B}) for arbitrarily
small $\delta>0$, $\delta^{\prime}>0$ and $\epsilon>0$. The claim 
follows here as well. Theorem~\ref{t2} is proved for $n\leq 5$.

Analyzing the proof, we see that the case \eqref{eq:n-1} works
for arbitrary $n$, and yields the following assertion in the spirit 
of Theorem~\ref{t2}.

\begin{theorem}
	Let $n\geq 1$ be an integer, $\xi$ a real transcendental number
	and $\varepsilon>0$.
	Then there are inifinitely many integer polynomials $P$ of degree
	exactly $n$ and without linear factor over $\mathbb{Z}[T]$ 
	that satisfy $|P(\xi)|\leq H(P)^{-\wo+\varepsilon}$.
\end{theorem}

\section{Preparations for $n=6$ and $n=7$}  \label{tow}

We prepare the proof for $n=6$ and $n=7$.
Let $n\leq 7$. Let $P_k, Q_k$ and the polynomials $S_{\ell}$ and $R_{\ell}$
be as in the proof for $n\leq 5$
and assume for the moment we are in case 1 of this proof,
i.e. $\deg(P_k)<n$ and $\deg(Q_k)<n$. 
We again estimate $\sharp \textbf{I}$. For each $k$, we now split
\[
\textbf{I}= \textbf{I}^{(1)} \cup \textbf{I}^{(2)} 
\cup \textbf{I}^{(3)}
\]
where
\begin{align*}
\textbf{I}^{(1)}&=\{ \ell \in  \textbf{I}: 1\leq \deg(A_{\ell})\leq n-3 \}, \\ \textbf{I}^{(2)}&=\{ \ell \in  \textbf{I}: \deg(A_{\ell})= n-2 \}, \\
\textbf{I}^{(3)}&=\{ \ell \in  \textbf{I}: \deg(A_{\ell})= n-1 \}.
\end{align*}

For estimating $\textbf{I}^{(1)}$, 
resp. $\textbf{I}^{(3)}$, essentially
the methods for $n\leq 5$ regarding estimation of 
$\textbf{J}^{(1)}$, resp. $\textbf{J}^{(2)}$,
work. For $\textbf{I}^{(3)}$ the exact same
argument yields $\sharp \textbf{I}^{(3)}\ll H_k^{o(1)}$. 
If $\deg(A_{\ell})\leq n-3$, then
we now have $2 \deg(A_{\ell}) -1 \leq 2(n-3)-1\leq n<\wo$ for $n\leq 7$ 
and again from Lemma~\ref{lemur} verify there are at most $\ll \log H_k$ solutions to \eqref{eq:dd}. This yields $\sharp \textbf{I}^{(1)}\ll \log H_k$.

We are left to estimate $\textbf{I}^{(2)}$. So
let $\ell\in \textbf{I}^{(2)}$, that is the polynomial
$S_{\ell}=A_{\ell}B_{\ell}$ splits as
\begin{equation}  \label{eq:AB}
\deg(A_{\ell})=n-2, \qquad \deg(B_{\ell})=2.
\end{equation}

Assume we have shown only 
\begin{equation}  \label{eq:scharp}
\sharp \textbf{I}^{(2)}\ll H_k^{o(1)}
\end{equation}
such indices
occur.
Then in total $\sharp \textbf{I}\leq \sharp \textbf{I}^{(1)}+\sharp \textbf{I}^{(2)}+ \sharp \textbf{I}^{(3)}\ll H_k^{o(1)}$ and thus
$\sharp (\mathscr{L}\setminus \textbf{I})= \sharp \mathscr{L}-\sharp \textbf{I}\gg H_{k}^{\delta}/ \log H_{k}$ many $\ell\in \mathscr{L}$ remain where
$S_{\ell}$ is irreducible resp. $\ell\in \mathscr{R}$ where $R_{\ell}$
irreducible, i.e. Theorem~\ref{t3} holds for $n=6$ and $n=7$.
Then we may again take any two such indices in $\mathscr{L}\setminus \textbf{I}$ and
conclude with Theorem~\ref{wbs} as in the last paragraph of 
\S~\ref{case1}.

It will be more convenient to write
$\textbf{I}_1=\textbf{I}^{(2)}$ in the sequel.
We prove \eqref{eq:scharp} indirectly.
We fix arbitrary $\gamma>0$, assume $\sharp \textbf{I}_1>H_k^{\gamma}$ for infinitely many $k$, and lead this to a contradiction
to finish the proof. For convenience
we prefer to separately consider $n=6$ and $n=7$, however we want to 
remark that it would be feasible to combine the arguments from \S~\ref{n6} and \S~\ref{s7} below to cover both cases at once.

\section{Proof of Theorems~\ref{t1},~\ref{t2} for $n=6$}  \label{n6}

We start from the observations of the last section.
So fix large $k$, assume $\deg(P_k)<n$, and \eqref{eq:AB} holds for a large set
of indices $\ell\in \textbf{I}_1$ with $\sharp\textbf{I}_1>H_k^{\gamma}$.
By a pigeon hole principle argument similar to the proof of  
Lemma~\ref{lemur},
we can show there are many indices within $\textbf{I}_1$
for which the respective induced $A_{\ell}$ have roughly the same height.

\textbf{Claim 1}: Let $\varepsilon>0$.
There are still $\gg_{\varepsilon} H_{k}^{\gamma}/ \log H_k$ many 
primes $\ell\in \textbf{I}_1$ with 
\begin{equation} \label{eq:F}
\tilde{X} \leq H(A_{\ell}) \leq \tilde{X}^{1+\varepsilon},
\end{equation}
for some fixed $\tilde{X}>1$.

We prove the claim. Recall the upper 
bound \eqref{eq:geller} for $H(A_{\ell})$, and let $c=c(n)$ be the implied constant. Accordingly,
we take any $X_0>1$ not too small and partition
$[X_0,cH_{k}^{1+\delta}]$ into successive intervals $J_1,\ldots,J_{\sigma(k)}$
of the form $[X,X^{1+\varepsilon}]$ starting at $X=X_0$, i.e.
\[
J_{i}= [X_0^{(1+\varepsilon)^{i-1}}, X_0^{(1+\varepsilon)^{i}}],\qquad\quad
1\leq i\leq \sigma(k),
\]
that pairwise have at most an endpoint in common. 
We may alter the right endpoint of the last $J_{\sigma(k)}$ 
to $H_k$ if needed, making it smaller. Obviously
\[
\sigma(k)\ll_{\varepsilon} \log_{1+\varepsilon} 
(\log(cH_{k}^{1+\delta})/\log X_0)\ll_{\varepsilon} \log\log H_{k}.
\]
Hereby we used $\delta<1$ to get rid of the dependence on $\delta$.
The total number of $\ell$ in $\textbf{I}_1$ is $\sharp \textbf{I}_1>H_{k}^{\gamma}$, and there
are possibly at most $\leq (2X_0+1)^{n+1}\ll_{n} 1$ remaining polynomials of
height $\leq X_0$ i.e. not in any such interval $J_{i}$. Hence,
by pigeon hole principle, there must 
be $\gg \sharp \textbf{I}_1/ \sigma(k)\gg_{\varepsilon} H_k^{\gamma}/ \log\log H_k$
indices $\ell\in \textbf{I}_1$ with $H(A_{\ell})$ in the same interval $J_{i}=:J$, for some $1\leq i\leq \sigma(k)$.
The claim follows by taking $\tilde{X}$ the left 
endpoint of $J$.
For the sequel we may consider $\varepsilon= (\wo/n-1)/2>0$ 
fixed and $\delta$ small enough fixed
and the implied constant in Claim~1 absolute.

Restrict to indices $\ell$ of Claim~1 satisfying \eqref{eq:F} 
and call the index set $\textbf{I}_2\subseteq \textbf{I}_1$ which we have seen has cardinality $\sharp \textbf{I}_2 \gg H_k^{\gamma}/ \log H_k$.
Write
\[
A_{\ell}(T)= a_{n-2}T^{n-2} + \cdots+ a_{0},\qquad \ell \in \textbf{I}_2,
\]
where the coefficients $a_i=a_i(\ell)$ depend on $\ell$ but for every $\ell$ we
have $a_{n-2}\neq 0$ and $a_{0}| d_{0}$ where $d_0= Q_k(0)=S_{\ell}(0)$ is the
constant coefficient of $Q_k$, since $T^{n-u_k}P_{k}(0)=0$. 
By the argument of \S~\ref{case2}
we can assume $d_0\neq 0$, so that by $|d_0|\leq H(Q_k)<cH_k<H_k$
and Lemma~\ref{rop}, including sign, again it 
has only $2\tau(d_0)\ll |d_0|^{o(1)} \leq H_{k}^{o(1)}$ many divisors.
Hence by pigeon hole principle there are still 
$\sharp \textbf{I}_2/ (2\tau(d_0))\gg \sharp \textbf{I}_2/ H_k^{o(1)}
\gg H_{k}^{\gamma-o(1)}$ many $\ell\in \textbf{I}_2$ inducing the same constant
coefficient $a=a_0$ of $A_{\ell}$. Call this index subset $\textbf{I}_3\subseteq \textbf{I}_2$ of cardinality $\sharp \textbf{I}_3 \gg H_{k}^{\gamma-o(1)}$
and restrict to such indices $\ell\in \textbf{I}_3$ in the sequel.
For $R_{\ell}$ instead, we have $a_{n-2}|c_n$ and $c_0\ne 0$, and the same 
estimate $\sharp \textbf{I}_3 \gg H_{k}^{\gamma-o(1)}$ is inferred
for the accordingly altered set $\textbf{I}_3\subseteq \textbf{I}_2$ of polynomials with the same fixed leading coefficient (for simplicity we use the same notation for these sets).

For simplicity, fix $\ell_0\in \textbf{I}_3$ that 
maximizes $H(A_{\ell})$ among
$\ell\in \textbf{I}_3$ and write $A=A_{\ell_0}$ and $H=H(A)= \max_{\ell \in \textbf{I}_3} H(A_{\ell})$.
For any other $\ell \neq \ell_0$ in $\textbf{I}_3$ we can consider 
for $S_{\ell}$ the polynomial
\[
G_{\ell}(T)= G_{\ell_0, \ell}(T) = \frac{ A(T) - A_{\ell}(T) }{T} \in \mathbb{Z}[T],
\]
and for $R_{\ell}$ instead
\[
G_{\ell}(T)= G_{\ell_0, \ell}(T) =  A(T) - A_{\ell}(T) \in \mathbb{Z}[T].
\]
In either case, by \eqref{eq:AB} it satisfies
\begin{equation}  \label{eq:dregs}
\deg(G_{\ell})\leq (n-2)-1=n-3, \qquad H(G_{\ell})\leq H+H(A_{\ell})\leq 2H,
\end{equation}
and by \eqref{eq:horse} and \eqref{eq:F} we can estimate
\begin{equation}  \label{eq:ungnad}
|G_{\ell}(\xi)| \ll_{\xi} \max \{ |A(\xi)| , |A_{\ell}(\xi)|  \}
\ll H^{-\wo+\varepsilon_2}\ll_{n} H(G_{\ell})^{-\wo+\varepsilon_3}. 
\end{equation}
Here and below all $\varepsilon_i$ are positive and arbitrarily small
as soon as $\varepsilon, \epsilon, \delta$ are small enough. 
Moreover note that all $G_{\ell}$
are pariwise distinct for different $\ell\in \textbf{I}_3$ and 
not identically $0$, since $A_{\ell}$ divides $S_{\ell}$ and by Proposition~\ref{aktiv}.

Now the irreducible $U_1(T)= A(T)$ and the polynomial
$U_2(T)=G_{\ell}(T)$ are clearly coprime by Proposition~\ref{aktiv},
and of degrees 
at most $n-2$ and $n-3$ resp. for all $\ell\in \textbf{I}_3$.
On the one hand, from \eqref{eq:horse}, \eqref{eq:geller}, \eqref{eq:dregs}, \eqref{eq:ungnad} we see
\begin{equation}  \label{eq:1111}
 H(U_i)\ll H^{1+\varepsilon_4}, \qquad |U_{i}(\xi)| \ll H^{-\wo+\varepsilon_5}, \qquad\qquad i=1,2.
\end{equation}
On the other hand, from Lemma~\ref{sb}
we get 
\begin{equation} \label{eq:stolt}
\max_{i=1,2} |U_{i}(\xi)| \gg_{n,\xi} H^{-(n-2)-(n-3)+1}= H^{-(2n-6)}.
\end{equation}
Combining with \eqref{eq:1111} yields $2n-6\geq (\wo-2\varepsilon_{5})/(1+\varepsilon_4)$ when we assume $H$ is sufficiently large. However, for $n\leq 6$ and small enough $\delta$ (thus $\varepsilon_4, \varepsilon_5$), this contradicts the assumption $\wo>n$. 
We should remark that potentially it can happen that $G_{\ell}$ 
is constant in which case we cannot apply Lemma~\ref{sb}. However,
since it cannot be identically $0$ (by Proposition~\ref{aktiv})
we then have the bound $\max_{i=1,2} |U_{i}(\xi)|\geq 1$ and the argument works as well since $\wo\ge n>0$.

Thus we have shown Theorem~\ref{t3} for $n=6$ in case 1.
The reduction of the other cases to case 1
and the deduction of Theorem~\ref{t1} via Theorem~\ref{wbs} is done 
precisely as for $n\leq 5$. The deduction of Theorems~\ref{t2}
follows analogously to $n\leq 5$ as well.

\section{Proof of Theorems~\ref{t1},~\ref{t2} for $n=7$} \label{s7}

Again we start from \S~\ref{tow} and want to contradict
$\sharp \textbf{I}_1\gg H_k^{\gamma}$.
We again consider the set $\textbf{I}_3\subseteq \textbf{I}_1$ 
from \S~\ref{n6}, that by the same line of arguments would 
still satisfy $\sharp \textbf{I}_3\gg H_k^{\gamma-o(1)}$.
First consider the polynomials $S_{\ell}$.
Write now $a_{n-2,\ell}$ for the leading coefficient of any $A_{\ell}$
for $\ell\in \textbf{I}_3$
to highlight the dependence on $\ell$. 
Let $c_{k,t}$ be the leading coefficient of 
$P_{k}(T)=c_{k,t}T^{t}+c_{k,t-1}T^{t-1}+\cdots+c_{k,0}$
of degree $t=t(k)<n$.
First notice that $a_{n-2,\ell}$ divides
$\ell c_{k,t}$, , for every $\ell$, since we assume $\deg(Q_k)<n$. 
We can assume $c_{k,t}>0$. Hence, since $\ell$ is prime, 
either $a_{n-2,\ell}|c_{k,t}$
or $a_{n-2,\ell}= \ell g_{\ell}$ where
$g_{\ell}|c_{k,t}$. 
Write $h_{\ell}$ for either $a_{n-2,\ell}$ or $g_{\ell}$ in the respective
cases. Notice the set of distinct $s_\ell$
that may appear in total has cardinality 
$\leq 2\tau(c_{k,t})\ll c_{k,t}^{o(1)}\leq H_k^{o(1)}$ as $k\to\infty$
by Lemma~\ref{rop}. Thus by pigeon hole principle
there remains a set $\textbf{I}_4 \subseteq \textbf{I}_3$
with cardinality $\sharp \textbf{I}_4 \geq \sharp \textbf{I}_3/ H_{k}^{o(1)}\gg H_k^{\gamma-o(1)}$ with the property that
both the constant coefficient $a_{0}=a_{0, \ell}$ and $h_{\ell}=:h_0$ 
coincide within the class $\ell\in \textbf{I}_4$.

Now consider $R_{\ell}$ and recall that now these polynomials $R_{\ell}$ have the
	same leading coefficient $a_{n-2,\ell}$ for all $\ell\in \textbf{I}_3$.
    Moreover $a_{0,\ell}$ divides $\ell d_0$ since $n-u_k>0$ implies $R_{\ell}(0)=\ell Q_k(0)=\ell d_0$. Recall $d_0\ne 0$ by \eqref{eq:wemay}.
    Thus again either $a_{0,\ell}|d_0$ or $a_{0,\ell}=\ell g_{\ell}$ 
    where $g_{\ell}|d_0$, and again writing $h_{\ell}$ for either
    $a_{0,\ell}$ or $g_{\ell}$,
     an analogous argument yields that at least
     $\gg H_k^{\gamma-o(1)}$ many $h_{\ell}$ appear in total.  
Again both the constant coefficients and $h_{\ell}$ 
coincide for all polynomials within $\textbf{I}_4$.  

Let us distinguish two cases.

\textbf{Case A:} There exist $\ell_1, \ell_2 \in \textbf{I}_4$ 
where the first case above applies, i.e. $h_{\ell_i}=a_{n-2,\ell_i}$ divides $c_{k,t}$ for the polynomials $S_{\ell}$ resp. 
$h_{\ell_i}=a_{0,\ell_i}$ divides $d_{0}$ for the polynomials $R_{\ell}$,
for $i=1,2$.
Then, recalling $\deg(A_{\ell})=n-2$, it is not hard to see that 
for either $S_{\ell}$ or $R_{\ell}$, the polynomial 
\[
G_{\ell_1,\ell_2}(T)= \frac{ A_{\ell_1} - A_{\ell_2 }}{T}\in \mathbb{Z}[T]
\] 
has degree at most $n-4$ (and it is impossible for $n\le 3$ by Proposition~\ref{aktiv}). 
Moreover, very similar to \eqref{eq:ungnad}, 
from \eqref{eq:horse}, \eqref{eq:F} with $H=\max_{i=1,2} H(A_{\ell_i})$ we again see that 
\[
|G_{\ell_1,\ell_2}(\xi)| \ll_{\xi} \max \{ |A_{\ell_1}(\xi)| , |A_{\ell_2}(\xi)|  \} \ll H^{-\wo+\varepsilon_2}.
\]
Hence, with $U_1=G_{\ell_1,\ell_2}, U_2=A_{\ell_1}$ we get
\eqref{eq:1111} again.
On the other hand, as again $U_1, U_2$ clearly have no common factor
by Proposition~\ref{aktiv}, 
from Lemma~\ref{sb} we now deduce
\[
\max_{i=1,2} |U_{i}(\xi)| \gg_{n,\xi} H^{-(n-2)-(n-4)+1}= H^{-(2n-7)}.
\]
Similar as in \S~\ref{n6} we can easily deal with the case $\deg(U_i)=0$
for some $i\in\{1,2\}$ separately, where the bound $\max_{i=1,2} |U_{i}(\xi)|\geq 1$ applies.
Combining the two estimates, 
again for large $H$ (or equivalently large $k$) we get a contradiction
to $\wo>n$ if $n\leq 7$ and $\varepsilon_2$ is small enough. 

\textbf{Case B:} For at most one $\ell \in \textbf{I}_4$ we have that
 $a_{n-2,\ell}$ divides $c_{k,t}$. Take any two distinct 
 indices $\ell_1, \ell_2 \in \textbf{I}_4$
 from the complement, so that $a_{n-2,\ell_i}= \ell_i h_{\ell_i}$ for $i=1,2$.
 Then, regarding $S_{\ell}$ consider now simultaneously the polynomials
 \[
 U_{1}(T)=G_{\ell_1, \ell_2}(T), \qquad U_{2}(T)= F_{\ell_1,\ell_2}(T):= \ell_2 A_{\ell_1} - \ell_1 A_{\ell_2 }\in \mathbb{Z}[T],
 \] 
and for $R_{\ell}$ instead
 \[
 U_{1}(T)=A_{\ell_1}(T)- A_{\ell_2}(T), \qquad
 U_2(T):=F_{\ell_1,\ell_2}(T):= \frac{\ell_2 A_{\ell_1} - \ell_1 A_{\ell_2 }}{T}\in \mathbb{Z}[T].
 \]
 Notice in the first case (for $S_{\ell}$) by construction the 
 leading coefficient with index $n-2$ vanishes.
 Thus in either case, both $U_i$ have degree at most $n-3$ and it is easy to verify, by
 a similar argument as in the proof of Proposition~\ref{aktiv}, that $U_1$ and $U_2$ have
 no non-constant common factor.
 Since $\ell_i\leq H_k^{\delta}$, we can estimate the height 
 of the polynomial $U_2$ via
 \begin{equation}  \label{eq:FRI}
 H(U_2) \leq \ell_2 H(A_{\ell_1})+ \ell_1 H(A_{\ell_2})
 \ll H_k^{\delta}\cdot \max_{i=1,2} H(A_{\ell_i}).
 \end{equation}
 The two different height notions $H_k$ and $H(A_{\ell_i})$ on the right hand side are disturbing, so
assume for the moment we have shown the following claim relating 
these quantities, to be justified
at the end of this section.

\textbf{Claim 2}: 
There exists $\Lambda>0$ such that for all large $k$ and $\ell\in \textbf{I}_4\setminus \{ \tilde{\ell}\}$ 
up to at most one exception $\tilde{\ell}=\tilde{\ell}(k)$ for each $k$,
we have at least one of the inequalities
\[
H(A_{\ell}) \geq H_k^{\Lambda}, \quad \text{or} \quad H(B_{\ell}) \geq H_k^{\Lambda}.
\]

We may assume the left inequality holds, the argument for the right
is very similar noticing that $B_{\ell}$ is irreducible as well
since it is quadratic and we can assume $S_{\ell}=A_{\ell}B_{\ell}$ has no linear
factor by Theorem~\ref{t4} (upon slightly shrinking $\textbf{I}_4$ if needed).
We may further assume $\ell_1, \ell_2$ are distinct from $\tilde{\ell}$
in Claim 2. 
Then, choosing $\delta$ small enough compared to $\Lambda$, from \eqref{eq:FRI} for arbitrarily small $\varepsilon_6>0$
and large $k$, we can guarantee
\[
H(U_2) \ll (\max_{i=1,2} H(A_{\ell_i}))^{1+\frac{\delta}{\Lambda}}\ll (\max_{i=1,2} H(A_{\ell_i}))^{1+\varepsilon_6}.
\]
Moreover, from \eqref{eq:horse} we see
\[
| U_{2}(\xi)| \leq \ell_2 |A_{\ell_1}(\xi)| + \ell_1 |A_{\ell_2}(\xi)|
\leq 2H_k^{\delta}\max_{i=1,2} |A_{\ell_i}(\xi)| 
\ll_{\xi} (\max_{i=1,2} H(A_{\ell_i}))^{ -\wo + 2\delta^{\prime}+ \frac{\delta}{\Lambda} }.
\]
Similar, in fact stronger, estimates apply to $U_1$.
Hence for our $U_1, U_2$, with $H=\max_{i=1,2} H(U_i)$,
again we verfiy \eqref{eq:1111} when taking small
enough $\delta>0$. On the other hand, from Lemma~\ref{sb} we again get
\[
\max_{i=1,2} |U_{i}(\xi)| \gg_{n,\xi} H^{-(n-3)-(n-3)+1}= H^{-(2n-7)}.
\]
Again we can easily deal with the case $\deg(U_i)=0$
for some $i\in\{1,2\}$ separately, then $\max_{i=1,2} |U_{i}(\xi)|\geq 1$.
Now combining, for large $H$ we get a contradiction
to $\wo>n$ if $n\leq 7$. Thus Theorem~\ref{t3} is shown for $n=7$ in case 1,
and the reduction of the other case 2, 3, 4 to case 1
and the deduction of Theorem~\ref{t1} via Theorem~\ref{wbs} is done 
precisely as for $n\leq 5$ again. The deduction of Theorem~\ref{t2}
follows analogously to $n\leq 5$ as well.

We finish by verifying Claim 2. Assume that for 
arbitrarily small given $\Lambda>0$ and certain large $k$, for 
$\ell\in \textbf{I}_4 \setminus \{ \tilde{\ell}\}$ 
we have
\begin{equation}  \label{eq:Kk}
H(A_{\ell}) < H_k^{\Lambda}.
\end{equation}
First notice that we can assume, again up to one possible exceptional 
index $\tilde{\ell}$, that
\[
| A_{\ell}(\xi)| \gg_{n,\xi} H(A_{\ell})^{-2n}.
\]
Indeed, for arbitrarily small
$\varepsilon_7>0$ 
since the $A_{\ell}$ are pairwise coprime by Proposition~\ref{aktiv},
Lemma~\ref{sb} and \eqref{eq:F} yield that for any two distinct $\ell$
at least one must satisfy
\begin{equation} \label{eq:beggin}
| A_{\ell}(\xi)| \gg_{n,\xi} H(A_{\ell})^{-2(n-2)+1 - \varepsilon_7}
= H(A_{\ell})^{-2n+5 - \varepsilon_7} > H(A_{\ell})^{-2n}.
\end{equation} 
If $n\le 5$, for small enough $\varepsilon_7>0$ we get a direct
contradiction to \eqref{eq:dd} since then $\theta>2n-5+\varepsilon_7$
as $\theta>n$ is fixed. Thus we may assume $n\ge 6$.
For any $\varepsilon_8>0$, which we can choose unaffected
from the choice of $\delta$, choosing $\Lambda< \varepsilon_{8}/(2n)$,
combining \eqref{eq:Kk}, \eqref{eq:beggin}, we get
\begin{equation}  \label{eq:melk}
| A_{\ell}(\xi)| \gg_{n,\xi} H_k^{- 2n\cdot \Lambda  }
\gg H_k^{-\varepsilon_8}.
\end{equation}
However, this means that the quadratic cofactor $B_{\ell}$ must have very small evaluation at $\xi$. Concretely, from \eqref{eq:kuh}, \eqref{eq:melk} and since $H(B_{\ell})\ll_{n} H(S_{\ell})\ll H_k^{1+\delta}$ by Gelfond's Lemma~\ref{gel},
for any $\varepsilon_9>0, \varepsilon_{10}>0$ and small 
enough $\delta>0$ we get
\[
| B_{\ell}(\xi)|=\frac{|S_{\ell}(\xi)|}{|A_{\ell}(\xi)|} \ll_{n,\xi} H_k^{-\wo+\varepsilon_9}\ll_{n} H(B_{\ell})^{-\wo+\varepsilon_{10}},
\qquad \ell\in \textbf{I}_4.
\]
	Note that this is the same estimate we have for $A_{\ell}$. On the
	other hand, Gelfond's Lemma~\ref{gel} and \eqref{eq:001a}, \eqref{eq:Kk}, when choosing $\Lambda<1/4$ imply 
\[
H(B_{\ell})=H(S_{\ell}/A_{\ell})\gg_n \frac{H(S_{\ell})}{H(A_{\ell})}\gg H_k^{1-\Lambda}>H_k^{1/2}>H_k^{\Lambda},
\]
and for $R_{\ell}$ similarly from Gelfond's Lemma~\ref{gel} and \eqref{eq:001b}, \eqref{eq:Kk} we get
\[
H(B_{\ell})=H(R_{\ell}/A_{\ell})\gg_n \frac{H(R_{\ell})}{H(A_{\ell})}\gg H_k^{1/2-\Lambda}>H_k^{\Lambda}.
\]

\begin{remark}  \label{r2}
		Analyzing the proof for $n=7$ in the sections above, we see that we only used $\wo>\max\{ 2n-7,0\}$, a consequence of $\wo>n$. Indeed, take $P,Q$ any polynomials as in Case 1 and without common factor of height at most $H$ satisfying
	\[
	\max\{ |P(\xi)|, |Q(\xi)|\} \le H^{-\mu},
	\]
	for some $\mu>\max\{ 2n-7,0\}$.
	 Indeed, using $\mu>2n-7$, the cases $\ell\in\textbf{I}^{(1)}$ and $\ell\in\textbf{I}^{(3)}$
	work precisely as in~\S~\ref{small}-\ref{n5}. When $\ell\in\textbf{I}^{(2)}$
	the only place below where we would require a stronger property
	is regarding \eqref{eq:stolt} in~\S~\ref{n6}. However, this part may be omitted as a refined argument superseeding~\S~\ref{n6} is provided in~\S~\ref{s7}. Indeed 
	the argument in~\S~\ref{s7} starts right after the paragraph in~\S~\ref{n6} where $\textbf{I}_3$ is defined.
	So, in the most critical case $\ell\in \textbf{I}^{(2)}$,
	we proceed as in \S~\ref{s7} and for analogously defined polynomials $U_1, U_2$, using Lemma~\ref{sb} we end up with the condition 
	\begin{equation} \label{eq:ipp}
	\max_{i=1,2} |U_{i}(\xi)| \gg_{n,\xi} H^{-(2n-7)}.
	\end{equation}
	On the other hand, as in the proof above 
	we may estimate the left hand side from 
	above by $\ll H^{-\mu+\epsilon}$, for small enough $\ell$ in terms
	of any given $\epsilon>0$.
	So as before, we require $\mu>2n-7$ for the contradiction.
	Again, in the case that some $U_i$ is constant, which implies $n<7$, where we cannot apply
	Lemma~\ref{sb} in the deduction, we may estimate $\max |U_i(\xi)|\ge 1$ in place of \eqref{eq:ipp}. Then we see that $\mu>0$ suffices for the contradiction.
Thus, a safe condition for all cases is $\mu>\max\{ 2n-7,0\}$.
\end{remark}

\begin{remark} \label{r3}
	A slight modification of the arguments in \S~\ref{s7} yields
	another proof of Theorem~\ref{H} for $n\leq 5$, or more
	generally for the case $\wo>2n-5$, that avoids case distinctions
	and the restriction to primes $\ell$
	(however, these conditions are still needed for $n\in\{6,7\}$).
	Then the most intricate case is \eqref{eq:n-1}, where 
	notably we cannot apply Theorem~\ref{t4} if $\ell$ is not prime.
	However, this case can
	be handled by considering the polynomials $G_{\ell_1,\ell_2}$
	and $F_{\ell_1,\ell_2}$ from~\S~\ref{s7} with the above argument.
	A minor adaption of the proof (slightly redefining
	$G_{\ell_1,\ell_2}, F_{\ell_1,\ell_2}$) is needed when not restricting $\ell$ to be prime. 
	In fact, we can even start
	with polynomials of the form 
	$M_{\ell_1, \ell_2}=\ell_1 T^{n-u_k}P + \ell_{2}Q$
	with coprime index pair, in place of $S_{\ell}$ or $R_{\ell}$. 
	We only sketch the adaptions to be made to obtain the last
	result, which will be the key to the proof of Theorem~\ref{t6}, (iii). Firstly, the analogue
	of \eqref{eq:hase} applies to  $M_{\ell_1, \ell_2}$ as well.
	Moreover, a variant of Proposition~\ref{aktiv} for $M_{\ell_1,\ell_2}, M_{\ell_1^{\prime},\ell_2^{\prime}}$ still
	works as soon as the pairs $(\ell_1, \ell_2)$ and $(\ell_1^{\prime},\ell_2^{\prime})$ are linearly independent.
	Next we derive an irreducible divisor $A_{\ell_1,\ell_2}$ of $M_{\ell_1,\ell_2}$ very similarly that 
	satisfies the estimate \eqref{eq:horse}. We again get
	many index pairs for which these polynomials
	have roughly the same
	height, as in the claim in \S~\ref{n6}. Furthermore,
	we still have that the constant coefficient of any $A_{\ell_1,\ell_2}$
	divides $\ell_2 d_0=\ell_2 Q_k(0)$ and its leading coefficient divides $\ell_1 c_{k,t}$ with $c_{k,t}$ 
	the leading coefficient of $P_k$.
	For any such pair
	$A_{\ell_1,\ell_2}, A_{\ell_1^{\prime},\ell_2^{\prime}}$
	we then again consider suitable integer linear combinations 	
	with small integers
	(and dividing by $T$ where applicable) to derive polynomials
	$G_{\ell_1,\ell_2,\ell_1^{\prime},\ell_2^{\prime} }$ and $F_{\ell_1,\ell_2,\ell_1^{\prime},\ell_2^{\prime}}$ 
	in place of $G_{\ell_1, \ell_2}$ and $F_{\ell_1, \ell_2}$.
	In the analogous case distinction,
	now these have degrees $\deg(G_{\ell_1,\ell_2,\ell_1^{\prime},\ell_2^{\prime} })\le n-3$  and $\deg(F_{\ell_1,\ell_2,\ell_1^{\prime},\ell_2^{\prime} })\le n-1$ in Case A 
	resp. both degrees are at most $n-2$ in Case B, due to the fact that we now assume $\deg(A_{\ell_1,\ell_2})=n-1$ in place of $n-2$. In each case similarly we define coprime polynomials $U_1, U_2$.
	They now again satisfy for some small $\varepsilon=\varepsilon(\delta)>0$
	\[
	\max |U_i(\xi)|\ll H^{-\wo+\varepsilon},
	\]
	but on the other hand, using Lemma~\ref{sb},
	in both case A and case B we have
	\begin{equation} \label{eq:hexe}
	\max |U_i(\xi)|\gg H^{-(2n-5)-\varepsilon}.
	\end{equation}
	For small enough $\delta$ and thus $\varepsilon$, this will again give a contradiction to our assumption $\wo>2n-5$.
	Similar to the case $n=3$ for $\wo>2n-7$,
	a small twist of the argument occurs when some $U_i$ become constant,
	as then Lemma~\ref{sb} cannot be applied so \eqref{eq:hexe} does not apply. Indeed, if $\wo>2n-5$ and
	$n=2$, the polynomials $G_{\ell_1,\ell_2,\ell_1^{\prime},\ell_2^{\prime}}$ and $F_{\ell_1,\ell_2,\ell_1^{\prime},\ell_2^{\prime}}$ are again constant (this can only happen in case B again).
	However, again both $U_i$ cannot identically vanish, 
	so then we can use the trivial lower 
	bound $\max |U_i(\xi)|\ge 1$ in place of the stronger bound
	$H^{1-\varepsilon}$ from \eqref{eq:hexe}. This induces via Proposition~\ref{pop}
	the more restrictive condition $\theta_2=1$ in claim (iii) of Theorem~\ref{t6}. We may again alter the condition $\wo>2n-5$ in the spirit of Remark~\ref{r2} for 
	any $\mu>\max\{ 2n-5,0\}$, by the same type of arguments. 
	We omit the technical details. 
\end{remark}

\section{Proof of Theorem~\ref{t5} and Theorem~\ref{units} } \label{pt5}

Recall that	by Remark~\ref{r2}, we can apply the
argument to any pair $P,Q$ of coprime polynomials of degrees $n,m$ respectively
with $n>m$ and with $P(0)=0$,
whose evaluations at $\xi$ 
are both of modulus smaller than $H^{-(2n-7+\epsilon)}$, with $H=\max\{ H(P),\; H(Q)\}$.

\begin{proof}[Proof of Theorem~\ref{t5}]
Let $\varepsilon>0$. For simplicity first assume $m<n$.
By definition of the exponent $w_{=m}^{int}(\xi)$,
there exists a sequence of irreducible 
monic polynomials $P_k$ of degree $m$ with
\[
|P_k(\xi)| \leq 
H(P_k)^{- w_{=m}^{int}(\xi) + \varepsilon}.
\]
Write $H_k=H(P_k)$. Notice that here $\deg(P_k)=u_k=m<n$. Again the definition
of $\widehat{w}_{n-1}(\xi)$ guarantees
there is $Q_k$ of degree $f_k$ at most $n-1$, height $H(Q_k) < cH_k$ and with 
\[
|Q_k(\xi)| \ll H_k^{- \widehat{w}_{n-1}(\xi) +\varepsilon }.
\]
Choosing $c=c(n)$ from Gelfond's Lemma~\ref{gel} the polynomials $Q_k$ and $P_k$ have no common factor.
This implies that $R_{\ell}=R_{k,\ell}=T^{n-m}P_k+\ell Q_k$ for small 
$1\leq \ell\leq H_k^{\delta}$ as above satisfy
\[
H(R_{\ell})\ll H_k^{1+\delta}, \qquad
|R_{\ell}(\xi)|\ll H_k^{- \Omega_m  +\delta+\varepsilon}.
\]
Hence if we choose $\delta, \varepsilon$ small,
by our observation at the beginning of this section (essentially Remark~\ref{r2}), 
and since $\Omega_m>0$ is always true by \eqref{eq:preis}, 
as soon as
\[
\Omega_m > 2n-7, 
\]
i.e. \eqref{eq:roest} holds, the irreducibilty arguments
again work for many $R_{\ell}$.  Also notice since $u_k<n, f_k<n$
we are in case 1 of our distinction from \S~\ref{small}. 
Moreover, each $R_{\ell}$ is monic of degree exactly $n$ 
since $P_k$ is monic and
$\deg(Q_k)\leq n-1<n$. From the above \eqref{eq:weem} is readily verified.  
Estimate \eqref{eq:weem2} is inferred by  
Theorem~\ref{wbs}, with a very similar argument as in~\S~\ref{case1} above, where we need the lower bound $n-1$ in \eqref{eq:Roest} in view
of this restriction in Theorem~\ref{wbs}.
Finally if $m=n$, we cannot directly
use the above argument to exclude that many $R_{\ell}$ have
a linear factor over $\mathbb{Z}[T]$.
However, to settle this we can proceed as in \S~\ref{case2}
to reduce it to the case $m<n$, noticing that divisors
of monic polynomials are again monic.
Finally \eqref{eq:tanz}
follows from \eqref{eq:weem} if we let $m=n-1$ and trivially estimate 
$\widehat{w}_{n-1}(\xi)\geq n-1$.
\end{proof}

The case of units works in a similar way.

\begin{proof}[Proof of Theorem~\ref{units}]
We start with a sequence $P_k^{u}\in\mathbb{Z}[T]$ of degree $m$ 
polynomials 
that are irreducible, monic and 
with constant coefficient $\pm 1$
and satisfy
\[
|P_k^{u}(\xi)| \leq 
(H_k^u)^{- w_{=m}^{u}(\xi) + \varepsilon}, \qquad\qquad k\geq k_0,
\]
where $H_k^{u}=H(P_k^{u})$.
The main twist is to consider here polynomials of the form
\[
R_{\ell}^{u}(T)=R_{k,\ell}^{u}(T)=(T+1)^{n-m}P_k^{u}+\ell TQ_k^{u},\qquad
1\leq \ell\leq H_k^{\delta},
\]
where $Q_k^{u}$ minimizes $|Q_k^{u}(\xi)|$ among integer 
polynomial of degree at most $n-2$ with $0<H(Q_k^{u})<cH_{k}^{u}$,
for $c=c(n)$ from Gelfond's Lemma~\ref{gel} again.
It satisfies 
$|Q_k^{u}(\xi)|\leq (H_k^{u})^{-\widehat{w}_{n-2}(\xi)+\varepsilon }$
for large $k$, thus
for any $\ell$ as above
\begin{equation} \label{eq:ttt}
H(R_{\ell}^{u})\ll (H_k^{u})^{1+\delta}, \qquad
|R_{\ell}^{u}(\xi)|\ll (H_k^{u})^{- \Omega_m^{\prime} +\delta+\varepsilon}.
\end{equation} 
In the special case that $T+1$ divides $Q_k^{u}$, we replace
$Q_k^{u}$ by $\tilde{Q}_k^{u}$
obtained from dividing $Q_k^{u}$ through all its $T+1$ factors.
This changes neither $H(Q_k^{u})$ nor $|Q_k^{u}(\xi)|$ significantly,
for large $k$. Moreover then $(T+1)^{n-m}P_k^{u}$ and $TQ_k^{u}$ have no common linear factor, and 
hence $P_k^u$ and any $R_{\ell}^u$ are coprime.
Moreover, by construction any $R_{\ell}^{u}$ is monic with constant 
coefficient $\pm 1$, and of degree
exactly $n$.
Moreover, the induced $R_{\ell}^{u}$ are
pairwise coprime by a very similar 
short argument as in Proposition~\ref{aktiv}.
We need to guarantee irreducibility of $R_{\ell}^{u}$
for some values $\ell$. To exclude linear factors,
here we cannot use Theorem~\ref{t4} 
since for $T=0$ we have $(T+1)^{n-m}P_k^{u}(T)=P_{k}^{u}(0)=\pm 1\neq 0$. 
However, by the special form of the polynomials $R_{\ell}^{u}$,
clearly the only possible linear factors
are $T+1$ and $T-1$, and 
for given $k$ by the coprimality noticed above
each factor can appear at most for one value of $\ell$ (in fact
$T+1$ is impossible since $(P_{k}^{u}, Q_{k}^{u})=1$).
The cardinality of those indices $\ell$ inducing a splitting of
$R_{\ell}^u$ only into integer polynomial
factors of degree greater than one can be bounded exactly as
$\sharp \textbf{J}^{(2)}$ in \S~\ref{case1}. 
Hence, indeed again there are many $\ell$ for which $R_{\ell}^{u}$ is 
irreducible, as requested.
In view of the observation in the introductory paragraph of this section
(Remark~\ref{r2}) and \eqref{eq:ttt}, the claimed
estimates are deduced very similarly as those
for algebraic integers in Theorem~\ref{t5} above, where again the bound $n-1$ in \eqref{eq:Runitassu} comes from the condition in Theorem~\ref{wbs}.
\end{proof}

\section{Proof of Theorem~\ref{t6}}

Claim (i) follows from Theorem~\ref{t4}. 
For claim (ii), let $\xi= (\alpha+\beta)/2$. Then 
Proposition~\ref{pop} and \eqref{eq:bau} yield
\[
|P(\xi)| \ll_{n,\xi} H\cdot |\xi-\alpha| <
H\cdot |\alpha-\beta| \leq H^{-(2n-7)-\epsilon},
\]
and similarly
\[
|Q(\xi)| \ll_{n,\xi} H\cdot |\xi-\beta| <
H\cdot |\alpha-\beta| \leq H^{-(2n-7)-\epsilon}.
\]
In view of the introductory comments from~\S~\ref{pt5} (Remark~\ref{r2}), we may identify $\wo$ with $2n-7+\epsilon$ 
and take the according $\delta$
obtained via the proofs in previous sections
upon identifying $P(T)=T^{n-u_k}P_k(T)$ and $Q(T)=Q_{k}(T)$.
As by assumption $n-u_k>0$ and $\deg(Q_k)<n$
as well, we find ourselves in the setting of Case 1 in the 
case distinction of~\S~\ref{small}, where the irreducible
polynomials are indeed among the $S_{\ell}$ (i.e. we do not need
to ``escape'' to derived polynomials as in Cases 2,3,4).
Carefully analyzing the steps of the proof, 
in particular regarding the cardinality drop in every 
transition from some $\textbf{I}_j$ to the consecutive $\textbf{I}_{j+1}$,
we further check that for at 
most $\ll_n \tau(c_n)\tau(d_0)\log H\ll H^{o(1)}$ many prime indices $\ell$ up to $H^{\delta}$ 
the polynomial $S_{\ell}$ can be reducible. In particular
$\gg H_{k}^{\delta}/ \log H_k$ many $\ell$ remain where $S_{\ell}$
is irreducible.
For $R_{\ell}$ analogous arguments apply.
  For claim (iii), note that our condition 
  \eqref{eq:baux} and Proposition~\ref{pop} now 
  for $n\geq 3$ imply
\[
\max\{ |P(\xi)|,|Q(\xi)|\} \ll_{n,\xi} H\cdot |\xi-\alpha| <
H\cdot |\alpha-\beta| \leq H^{-(2n-5)-\epsilon},
\]
and similarly for $n=2$.
The proof then follows essentially from the 
procedure explained in Remark~\ref{r3}.
Finally, it is readily checked that the arguments work 
for $\xi\in\mathbb{C}$ since this
is true for the prerequisites Lemma~\ref{sb} and Proposition~\ref{pop}, 
so $\alpha$ and $\beta$ may be complex in claims (ii), (iii) as well.

\begin{remark}
	Minor modifications of the proof show that 
	alternative conditions to \eqref{eq:bau},
	\eqref{eq:baux} can be stated. For example for (ii) we may impose
	either $|Q(\alpha)|\leq H^{-\kappa_n+1-\epsilon}$ or $|P(\beta)|\leq H^{-\kappa_n+1-\epsilon}$ holds in place of \eqref{eq:bau}, and likewise for (iii).
\end{remark}

\section{Final comments} \label{fc}

The first case where our results are open is for $n= 8$. 
Concretely then with our argument we cannot exclude that $S_{\ell}=A_{\ell}B_{\ell}$, or $R_{\ell}=A_{\ell}B_{\ell}$,
factors as in \eqref{eq:AB} with an irreducible factor $A_{\ell}$ of degree
$6$ and quadratic irreducible $B_{\ell}$, for many $\ell$. 
Irreducibility
criteria as in our crucial auxiliary result Theorem~\ref{t3} are typically quite challenging, see the problems of Szegedy and Turan
discussed in \S~\ref{intro3}.
Our proof for $n\leq 7$ heavily relied 
on the fact that $P=P_k$ and $Q=Q_k$ are both small at some $\xi$, i.e.
have some close pair of roots.
We once again stress that when Theorem~\ref{t3} and thus
Theorem~\ref{t1} holds for some $n$,
then the bounds of Badziahin and Schleischitz~\cite{badsch} 
are applicable to $w_{=n}^{\ast}(\xi)$, which are larger than $n/\sqrt{3}$.
Just for small $n$ the bound in Theorem~\ref{H} turns out to be stronger.

\vspace{0.7cm}

{\em The author thanks the referee for the very careful reading and 
	for pointing out several inaccuracies and providing helpful suggestions. The author is further grateful to Damien Roy for advice, particularly 
	for pointing out Remark~\ref{r1}. }


\begin{thebibliography}{99}
	
	
	\bibitem{apostol} T.M. Apostol. Introduction to analytic number theory. {\em Undergraduate Texts in Mathematics}, 
	New York-Heidelberg, Springer press (1976).
	
	\bibitem{badz} D. Badziahin. Upper bounds for the uniform simultaneous Diophantine exponents. {\em  Mathematika} 68 (2022), no. 3, 805–-826. 
	
	 \bibitem{badsch} D. Badziahin, J. Schleischitz. An improved bound in Wirsing's problem. {\em Trans. Amer. Math. Soc.} 374 (2021), no. 3, 1847–-1861.
	 
	 \bibitem{hajdu} A.  B\'{e}rczes, L. Hajdu. Computational experiences on the distances of polynomials to irreducible polynomials. {\em Math. Comp.} 66 (1997), no. 217, 391–-398.
	 
	 \bibitem{bb} N.C. Bonciocat, Y. Bugeaud, M. Cipu, M. Mignotte. Irreducibility criteria for sums of two relatively prime polynomials. {\em Int. J. Number Theory} 9 (2013), no. 6, 1529-–1539.
	 
	 \bibitem{bugdraft} Y. Bugeaud. Exponents of Diophantine approximation. {\em Dynamics and analytic number theory}, 96–-135, London Math. Soc. Lecture Note Ser., 437, Cambridge Univ. Press, Cambridge, 2016.
	 
	 \bibitem{bugbuch} Y. Bugeaud.
	 Approximation by algebraic numbers. {\em Cambridge Tracts in Mathematics}
	 160, Cambridge University Press, Cambridge, 2004.
	 
	 \bibitem{buschlei} Y. Bugeaud, J. Schleischitz. On uniform approximation to real numbers.
	 {\em Acta Arith.} 175 (2016), 255--268.
	 
	 \bibitem{buteu} Y. Bugeaud, O. Teuli\'{e}. Approximation d'un nombre r\'{e}el par des nombres alg\'{e}briques de degr\'{e} donn\'{e}. (French) [Approximation of a real number by algebraic numbers of a given degree] {\em Acta Arith.} 93 (2000), no. 1, 77-–86.
	 
	 \bibitem{cavachi} M. Cavachi. On a special case of Hilbert's irreducibility theorem. {\em J. Number Theory} 82 (2000), no. 1, 96--99.
	 
	 \bibitem{davsh67} H. Davenport, W.M. Schmidt.
	 Approximation to real numbers by quadratic irrationals.
	 {\em Acta Arith.} 13 (1967), 169--176.
	 
	 \bibitem{davsh} H. Davenport, W.M. Schmidt.
	 Approximation to real numbers by algebraic integers.
	 {\em Acta Arith.} 15 (1969), 393--416.
	 
	 \bibitem{fil} M. Filaseta. Is every polynomial with integer coefficients near an irreducible polynomial? {\em Elem. Math.} 69 (2014), no. 3, 130–-143. 
	 
	 \bibitem{german} O. German. On Diophantine exponents and Khintchine's transference principle. {\em Moscow J. Combin. Number Theory} 2 (2012), 22--51.
	 
	 \bibitem{gy} K. Gy\H{o}ry. On the irreducibility of neighbouring polynomials. {\em Acta Arith.} 67 (1994), no. 3, 283–-294.
	 
	 \bibitem{jarnik} V. Jarn\'ik. Contribution \`a la th\'eorie des approximations diophantiennes lin\'eaires et homog\`enes. (Russian) {\em Czechoslovak Math. J.} 4(79) (1954), 330–-353. 
	 
	 \bibitem{KH} Y. A. Khintchine. \"Uber eine Klasse linearer Diophantischer Approximationen. {\em Rendiconti Palermo} 50 (1926), 170--195. 
	 
	 \bibitem{laurent} M. Laurent.  Simultaneous rational approximation to the successive powers of a real number. {\em Indag. Math.} (N.S.) 14 (2003), no. 1, 45--53.
	 
	 \bibitem{nimo} N. Moshchevitin. A note on two linear forms. {\em Acta Arith.} 162 (2014), 43--50.
	 
	 \bibitem{pr} A. Poels, D. Roy. Simultaneous rational approximation to successive powers of a real number. {\em  Trans. Amer. Math. Soc.} 375 (2022), no. 9, 6385–-6415. 
	 
	 \bibitem{royann} D. Roy. Approximation to real numbers by cubic algebraic integers. II. {\em Ann. of Math.} (2) 158 (2003), no. 3, 1081–-1087.
	 
	 \bibitem{roy} D. Roy. Approximation to real numbers by cubic algebraic integers. I. {\em Proc. London Math. Soc.} (3) 88 (2004), no. 1, 42--62.
	 
	 \bibitem{roy3} D. Roy. On simultaneous rational approximations to a real number, its square, and its cube. {\em Acta Arith.} 133 (2008), no. 2, 185--197.
	 
	 \bibitem{schi} A. Schinzel. Reducibility of polynomials and covering  systems of congruences. {\em Acta Arith.} 13 (1967), 91--101.

     
      \bibitem{ext} J. Schleischitz. Approximation to an extremal number, its square and its cube. {\em Pacific J. Math.} 287 (2017), no. 2, 485–-510.
      
     
     \bibitem{jnt} J. Schleischitz. Cubic approximation
     to Sturmian continued fractions. {\em
     	J. Number Theory}  184 (2018), 270--299.
     
     \bibitem{ichacta} J. Schleischitz. Uniform Diophantine approximation and best approximation polynomials. {\em Acta Arith.} 185 (2018), no. 3, 249–-274.
     
     \bibitem{moscj} J. Schleischitz. Diophantine approximation in prescribed degree. {\em Mosc. Math. J.}  18 (2018), no. 3, 491--516.
     
     \bibitem{equprin} J. Schleischitz. An equivalence principle between polynomial and simultaneous Diophantine approximation. {\em  Ann. Sc. Norm. Super. Pisa Cl. Sci.} (5) 21 (2020), 1063--1085. 
     
     \bibitem{period} J. Schleischitz. On geometry of numbers and uniform rational approximation to the Veronese curve. {\em  Period. Math. Hungar.} 83 (2021), no. 2, 233–-249.
     
     \bibitem{ss} W.M. Schmidt, L. Summerer. The generalization of Jarn\'ik's identity. {\em Acta Arith.} 175 (2016), no. 2, 119-–136.
     
     \bibitem{teu} O. Teuli\'{e}. Approximation d'un nombre r\'{e}el par des unités alg\'{e}briques. (French) [Approximation of a real number by algebraic units] {\em Monatsh. Math.} 132 (2001), no. 2, 169--176.
     
     
     \bibitem{Tsi07} K.I. Tsishchanka.
     On approximation of real numbers by algebraic numbers of bounded degree.
     {\em J. Number Theory} 123 (2007), no. 2, 290--314.
     
     \bibitem{wirsing} E. Wirsing. Approximation mit algebraischen Zahlen beschr\"ankten Grades.
     {\em J. Reine Angew. Math.} 206 (1961), 67--77.











\end{thebibliography}
\end{document}